\documentclass[12pt]{article}
\usepackage[cp1251]{inputenc}
\usepackage{amsmath,amsthm,mathrsfs}
\usepackage{amssymb}
\usepackage[english]{babel}
\usepackage{esint}
\usepackage[titletoc]{appendix}
\usepackage{bbm}
\usepackage{amsfonts,amstext,amssymb,verbatim,epsfig}%,psfrag}
\usepackage{pgf,tikz}
\usepackage{float}
\usetikzlibrary{arrows}
\usepackage{hyperref}
\usepackage{cite}
\usepackage{epstopdf}
\usepackage{color}
\usepackage{transparent}
\usepackage{subfigure}

\usepackage{xcolor}
\usepackage[normalem]{ulem}

\hypersetup{colorlinks=true,pdfborder=001,  citecolor  = blue, citebordercolor= {magenta}}
\hypersetup{linkcolor=magenta, linkbordercolor = blue}
\hypersetup{colorlinks,urlcolor=blue}
\makeatletter
\let\reftagform@=\tagform@
\def\tagform@#1{\maketag@@@{(\ignorespaces\textcolor{magenta}{#1}\unskip\@@italiccorr)}}
\renewcommand{\eqref}[1]{\textup{\reftagform@{\ref{#1}}}}
\makeatother

\makeatletter
\DeclareUrlCommand\ULurl@@{%
  \def\UrlLeft{\uline\bgroup}%
  \def\UrlRight{\egroup}}
\def\ULurl@#1{\hyper@linkurl{\ULurl@@{#1}}{#1}}
\DeclareRobustCommand*\ULurl{\hyper@normalise\ULurl@}
\makeatother

\def\lessim{\ \lower4pt\hbox{$
		\buildrel{\displaystyle <}\over\sim$}\ }
\def\gessim{\ \lower4pt\hbox{$\buildrel{\displaystyle >}
		\over\sim$}\ }

\newtheorem{theorem}{\bf Theorem}
\newtheorem{lemma}[theorem]{\bf Lemma}

\newtheorem{proposition}[theorem]{\bf Proposition}

\theoremstyle{remark}
\newtheorem{remark}{Remark}

\newenvironment{Proof of lemma}{\noindent{\bf Proof of Lemma}}{\hfill$\Box$\newline}
\newenvironment{Proof of theorem}{\noindent{\bf Proof of Theorem}}{\hfill{\footnotesize${\square}$}\newline}
\newenvironment{Proof of theorems}{\noindent{\bf Proof of Theorems}}{\hfill$\Box$\newline}
\newenvironment{Proof of proposition}{\noindent{\bf Proof of Proposition}}{\hfill$\Box$\newline}
\newenvironment{Proof of propositions}{\noindent{\bf Proof of Propositions}}{\hfill$\Box$\newline}
\newenvironment{Proof of exercise}{\noindent{\it Proof of Exercise:}}{\hfill$\Box$}
\begin{document}

\title{On properties of the spherical mixed vector $p$-spin model}

\author{Antonio Auffinger \thanks{Department of Mathematics, Northwestern University, tuca@northwestern.edu, research partially supported by NSF Grant CAREER DMS-1653552, Simons Foundation/SFARI (597491-RWC), and  NSF Grant 1764421.} \\
\small{Northwestern University}
	\and Yuxin Zhou 
	\thanks{Department of Mathematics, yuxinzhou2023@u.northwestern.edu}
	\\ \small{Northwestern University}
}

\maketitle

\begin{abstract}
This paper studies properties of the mixed spherical vector $p$-spin model. At zero temperature, we establish and investigate a Parisi type formula for the ground state energy. At finite temperature, we provide some properties of minimizers of the Crisanti-Sommers formula recently obtained in \cite{Ko}. In particular, we extend some of the one-dimensional Parisi measure results of \cite{1} to the vector case.
\end{abstract}

\section{ The spherical spin model with vector spins}

Spherical spin glass models are one of the main sources of ideas and techniques in the theoretical study of disordered complex systems. These models are simple enough to produce explicit computations while retaining many of the intriguing phenomena of high-dimensional random systems. Their energy landscape provides a metaphor to explain several phenomena in other areas of science, including biology, chemistry, data science, and economy. 

One of these explicit computations was the limiting free energy discovered by Crisanti and Sommers in \cite{2} for the spherical $p$-spin model with one dimensional spins. This formula is the analogue of the classical Parisi formula for the Sherrington-Kirkpatrick model \cite{SK} and it was rigorously  proven for even-$p$-spin models by Talagrand in \cite{6} and extended to general mixed $p$-spin models by Chen in \cite{CASS}. These variational formulas and their minimizers have deep importance to describe and classify the energy landscape of such systems. We refer the readers to \cite{AB, ABC, AJ, ParisiMeasure, 1, replicaonandoff, phasediagram, dualitygroundstate, MPV} and the references therein for results in this direction.

The rigorous study of spherical spin models with vector spins started with the work of Panchenko and Talagrand \cite{PTSPHERE} with the first non-trivial bounds for the free energy.  Recently, Ko\cite{Ko} provided a proof of the limiting free energy for these models and its Crisanti-Sommers analogue \cite{JK1}. These results came after important contribution of Panchenko in the study of vector $p$-spins on the hypercube \cite{4,5,PTSPHERE}. As far as we know, there is no rigorous study on the role and properties of the minimizers of such models. The goal of this paper is to provide the first steps of this study and to analyze the model at zero temperature, extending the results of \cite{Ko} to the ground state energy. Our main results are stated in Sections \ref{sec32} and \ref{Sec4}.

Let us now describe the spherical model with vector spins and state some of its fundamental results. Fix $m\geq 1$ and for $N \geq 1$, let $S_N$ be the sphere in $\mathbb{R}^N$ of radius $\sqrt{N}$. We denote a configuration of the vector spin by 
\begin{equation*}
\vec{\sigma}=(\vec{\sigma}_1,\cdots,\vec{\sigma}_N)\in  S_N^m \text{ where } S_N^m=\{\vec{\sigma} \in (\mathbb{R}^N)^m | \vec{\sigma}(j) \in S_N  \text{ for } j=1,\cdots,m\}.
\end{equation*}
Here the $j$-th coordinate of $\vec{\sigma}$ is denoted by $\vec{\sigma}(j)$ and the vector entries of $\vec{\sigma}$ are denoted by
\begin{equation*}
\vec{\sigma}_i=(\vec{\sigma}_i(1), \cdots, \vec{\sigma}_i(m)) \in \mathbb R^m, \quad 1 \leq i \leq N.
\end{equation*}

For $p \geq 2$, we denote the $p$-spin Hamiltonian of the $j$-th copy by 
\begin{equation*}
H_{N,p}(\vec{\sigma}(j))=\frac{1}{N^{\frac{p-1}{2}}}\sum_{1 \leq i_1, \cdots,i_p \leq N}g_{i_1,\cdots,i_p}\vec{\sigma}_{i_1}(j) \cdots \vec{\sigma}_{i_p}(j),
\end{equation*}
where $g_{i_1,\cdots,i_p}$ are i.i.d. standard Gaussians for all $p \geq 2$ and indices $(i_1, \cdots , i_p)$. The corresponding mixed $p$-spin Hamiltonian for the $j$-th copy at inverse temperatures $(\vec{\beta}_p)_{p \geq 2}$, where $\vec{\beta}_p=(\vec{\beta}_p(j))_{1 \leq j \leq m}$, can be expressed as 
\begin{equation}\label{Hamiltonian}
H^j_N(\vec{\sigma})=\sum_{p\geq 2} \vec{\beta}_p(j) H_{N,p}(\vec{\sigma}(j)).
\end{equation}
Here we only consider mixed even $p-$spin models, $i.e.$ $\beta_p(j)=0$ for all $1 \leq j \leq m$ and odd $p \geq 3$. Moreover, we assume that the inverse temperature of each $j$-th copy satisfy $$\sum_{p \geq 2} 2^p \vec{\beta}_p^2(j) < \infty,$$ so that \eqref{Hamiltonian} is well-defined.

We define the Hamiltonian of $m$ copies mixed p-spin models of spherical spin glasses by 
\begin{equation*}
H_N(\vec{\sigma})=\sum_{j=1}^m H_N^j(\vec{\sigma}).
\end{equation*}
If, for any $1\leq k,\ell \leq m$,  we introduce the function
\begin{equation}\label{eq:xi}
\xi_{k, \ell} (x)= \sum_{p\geq 2} \beta_p(k)\beta_p(\ell) x^p,
\end{equation}
then it is not difficult to check that, for two arbitrarily spin configurations $\vec{\sigma}^1$ and $\vec{\sigma}^2$,
\begin{equation*}
\mathbb{E} \left[ H_N^k(\vec{\sigma}^1(k))H_N^{\ell}(\vec{\sigma}^2(\ell)) \right]=N \xi_{k,\ell}(R^{k,\ell}_{1,2}) \text{ for all } 1 \leq k, \ell \leq m,
\end{equation*}
where $$R^{k,\ell}_{1,2}  = \frac{1}{N}\sum_{j=1}^N \vec{\sigma}^1_j(k)\vec{\sigma}^2_j(\ell)$$ is the overlap between the corresponding coordinates of the vector configurations $\vec{\sigma}^1$ and $\vec{\sigma}^{2}$. The overlap matrix between configurations $\vec {\sigma}^\ell$ and $\vec {\sigma}^{\ell'}$ is expressed as
\begin{equation*}
R_{\ell,\ell'}=R({\vec{\sigma}}^\ell,{\vec{\sigma}}^{\ell'})=(R^{kk'}_{\ell,\ell'})_{1 \leq k,k' \leq m}=\frac{1}{N}\sum_{i=1}^N \vec{\sigma_i}^\ell \otimes \vec{\sigma_i}^{\ell'}
\end{equation*}
where $\otimes$ is the outer product on vectors in $\mathbb{R}^m$

 Let $\mathbb M$ be the space of $m \times m$ semi-symmetric positive definite matrices with entries in $[-1,1]$ and diagonals entries equal to $1$.
We now describe the Crisanti-Sommers formula for the free energy of a system of vector spins with constrained self-overlap $ Q \in \mathbb M$. 
For any positive semidefinite matrix $A=(A_{i,j})_{1 \leq i,j \leq m}$, let 
\begin{equation*}
\xi(A):=\sum_{p \geq 2 }(\vec{\beta}_p \otimes \vec{\beta}_p) \odot A^{\circ p} = (\xi_{i,j}(A_{i,j}))_{1\leq i,j \leq m}
\end{equation*}
where $\xi_{i,j}$ are defined in \eqref{eq:xi} and $A^{\circ p}$ denotes the $p$-th Hadamard power of the matrix $A$ (element-wise multiplication). Clearly, we have the matrices 
\begin{equation*}
\xi'(A)=\sum_{p \geq 2 }p(\vec{\beta}_p \otimes \vec{\beta}_p) \odot A^{\circ (p-1)}  \text{ and }  \xi''(A)=\sum_{p \geq 2 }p(p-1)(\vec{\beta}_p \otimes \vec{\beta}_p) \odot A^{\circ (p-2)}.
\end{equation*}

Given any $\epsilon > 0$ and $Q\in \mathbb M$, denote the set of spins with constrained self overlaps by 
\begin{equation*}
 \mathcal Q_N^{\epsilon}= \left \{ \vec{\sigma} \in S^m_N \bigg | \; \| R(\vec{\sigma},\vec{\sigma}) - Q \|_{\infty} \leq \epsilon \right\},
\end{equation*}
where $\| A \|_{\infty}= \sup_{1 \leq i,j \leq m} |A_{ij}|$, for any $m \times m$ matrix $A$.
For an external field $\vec{h} \in \mathbb{R}^m$ and any $\beta >0$, we define the free energy as 
\begin{equation*}
F_N^{\epsilon,Q}(\beta)=\frac{1}{N} \mathbb{E} \log \int_{\mathcal Q_N^{\epsilon}} \exp \beta \left(H_N(\vec{\sigma})+\sum_{j=1}^m \vec{h}(j) \sum_{i=1}^N \vec{\sigma}_i (j)\right) d \lambda^n_N(\vec{\sigma}),
\end{equation*}
where the reference measure $\lambda^m_N=\lambda_N^{\otimes m}$ is the product of Haar measures $\lambda_N$ on $S_N$ with normalization $\lambda_N(S_N)=1$. Here the parameter $\beta$ is the so-called inverse temperature.
Moreover, denote $\xi_\beta=\beta^2 \xi$ and $\vec{h}_\beta=\beta \vec{h}$.

 For a measurable function $f:S_N^m\times S_N^m \to \mathbb R$ we also set 
\begin{equation}\label{eq:GM}
\langle f \rangle_{\epsilon, Q}=\frac{\int_{(\mathcal Q_N^{\epsilon})^2} f(\vec{\sigma}^1,\vec{\sigma}^2)\exp(H_N({\vec{\sigma}^1)}+H_N({\vec{\sigma}^2)}) d \lambda^m_N({\vec{\sigma}}^1)d \lambda^m_N({\vec{\sigma}}^2)}{\int_{(\mathcal Q_N^{\epsilon})^2} \exp(H_N({\vec{\sigma}^1)}+H_N({\vec{\sigma}^2)) d \lambda^m_N({\vec{\sigma}}^1)d \lambda^m_N({\vec{\sigma}}^2)}}.
\end{equation}
For  a matrix-valued function $A(\vec{\sigma}^1,\vec{\sigma}^2)=(a_{ij}(\vec{\sigma}^1,\vec{\sigma}^2))_{1 \leq i,j \leq m}$, we let $\langle A(\vec{\sigma}^1,\vec{\sigma}^2) \rangle_{\epsilon, Q}$ denote the matrix $(\langle a_{ij}(\vec{\sigma}^1,\vec{\sigma}^2) \rangle_{\epsilon, Q})_{1 \leq i,j \leq n}$.
Observe that for any $f$ continuous, the map 
\begin{equation}\label{cont:contin}
Q \mapsto \langle f \rangle_{\epsilon, Q}
\end{equation}
is continuous on $\mathbb (\mathbb M, \| \cdot \|_{\infty})$.

We will now recall the formula for the free energy obtained in \cite{JK1}. Denote a right-continuous non-decreasing function by 
\begin{equation}\label{eq:xpath}
x(t):[0,m] \rightarrow [0,1] \text{ such that } x(0)=0 \text{ and } x(m)=1 
\end{equation}
and a $1$-Lipschitz monotone matrix path in the space of $m \times m$ positive semidefinite matrices by 
\begin{equation}\label{eq:ppath}
\Psi(t):[0,m] \rightarrow \mathbb{S}^m_+ \text{ such that } \text{Trace}(\Psi(t))=t  \text{ and }  \Psi(0)=0 \text{ and } \Psi(m)=Q.
\end{equation}
where $\mathbb{S}_{+}^m$ is the space of $m \times m$ positive semidefinite matrices.

 Set 
\begin{equation*}
t_x:=x^{-1}(1)=\inf\{t \in [0,m]|x(t) =1\}
\end{equation*}
and
\begin{equation*} 
s_x:=x^{-1}(0)=\sup\{t \in [0,m]|x(t) =0\}.
\end{equation*}

 Assuming $t_x < m$, then $\forall t_x < \hat{T} < m$, we can define the quantity 
 \begin{equation}\label{CSF}
 \mathscr{C}_{\beta,Q}(x,\Phi)=\frac{1}{2} \left[ \int_0^m x(t) \big\langle \xi'_\beta(\Phi(t))+\vec{h}_\beta \vec{h}_\beta^T,\Phi'(t) \big\rangle dt+\log|{\Phi(m)-\Phi(\hat{T})}|+\int^{\hat{T}}_0 \big \langle \hat{\Phi}(t)^{-1},\Phi'(t) \big \rangle dt\right] 
 \end{equation}
where $\hat{\Phi}(t):[0,m] \rightarrow \mathbb{R}^{m \times m}$ is a decreasing matrix path given by 
\begin{center}
$\hat{\Phi}(t)=\int^m_t x(s)\Phi'(s)ds$.
\end{center}

Furthermore, for any $\Lambda \in \mathbb{S}^m_+$ satisfying $\Lambda > \int^m_0 x(s) \xi_\beta''(\Phi(s)) \odot \Phi'(s) ds $,we have a continuous form of the Parisi formula as follows:
\begin{eqnarray}\label{eq:contPAr}
\mathscr{P}_{\beta,Q}(x,\Lambda, \Phi)=\frac{1}{2} \bigg[ \int^m_0 \langle \xi_\beta''(\Phi(q)) \odot \Phi'(q), (\Lambda-D^x_{\beta,Q}(q))^{-1} \rangle dq+ \langle \vec{h}_\beta\vec{h}_\beta^T , ( \Lambda-D^x_{\beta,Q}(0))^{-1} \rangle \nonumber \\
- \int^m_0 x(q) \langle \xi_\beta''(\Phi(q))\odot \Phi(q),\Phi'(q) \rangle dq  + \langle \Lambda , Q \rangle -m - \log | \Lambda | \bigg]
\end{eqnarray}
where $D^x_{\beta,Q}(q):= \int^m_q x(s) \xi_\beta''(\Phi(s)) \odot \Phi'(s) ds$.

Denote $\mathscr{M}_0$ the collection of all 1-Lipschitz monotone matrix paths that satisfy \eqref{eq:ppath} and $\mathscr{M}$ the collection of all continuously differentiable Lipschitz monotone matrix path with Lipschitz derivatives that satisfies \eqref{eq:ppath} without satisfying the assumption that Trace$(\Phi(t))=t$, but satisfying the assumption that for any $t \in$ supp $\mu$, Trace$(\Phi(t))=t$.

Similarly, denote by $\mathscr{N}_0$ the collection of all nonnegative nondecreasing and right-continuous functions on $[0,m)$ and $\mathscr{N}$ the collection of all nonnegative nondecreasing and right-continuous functions on $[0,m)$ satisfying the assumption that there exists $0<c<m$ such that $\alpha(t)$ is constant for all $t \in [c,m]$.

We will also need a discrete version of \eqref{CSF} and \eqref{eq:contPAr} that we describe now. 
Consider a discrete monotone matrix path encoded by an increasing sequence of real numbers and a monotone sequence of $n \times n$ symmetric positive semidefinite matrices,\begin{eqnarray*}
0=x_0 \leq x_1 \leq \ldots \leq x_{r-2} \leq x_{r-1} \leq 1, \\
0=Q_0 \leq Q_1 \leq \ldots \leq Q_{r-2} \leq Q_{r-1} \leq Q_r =Q,
\end{eqnarray*}
where $r \geq 1$. We denote $\b{x}=(x_k)^{r-1}_{k=0}$ and $\b{Q}=(Q_k)^r_{k=1}$. The discrete Crisanti-Sommers formula is given by
\begin{eqnarray}\label{eq:discreteCS}
\mathscr{C}_r(\b{x},\b{Q})=\frac{1}{2} [\langle \vec{h}\vec{h}^T, D_1 \rangle + \frac{1}{x_{r-1}} \log |Q-Q_{r-1}| - \Sigma_{1 \leq k \leq r-2} \frac{1}{x_k} \log\frac{|D_{k+1}|}{|D_k|}  \nonumber \\
+ \langle Q_1,D_1^{-1} \rangle +\Sigma_{1 \leq k \leq r-1} x_k \cdot \text{Sum}(\xi(Q_{k+1})-\xi(Q_k)) ],
\end{eqnarray}
where $D_p=\Sigma_{p \leq k \leq r-1} x_k (Q_{k+1}-Q_k)$ for $1 \leq p \leq r-1$, and 
\[ \text{Sum}(A) = \sum_{1\leq i,j \leq m} A_{ij}.
\]
and the discrete Parisi formula is given by
\begin{eqnarray}\label{eq:discreteParisi}
\mathscr{C}_r(\Lambda,\b{x},\b{Q})&=&\frac{1}{2}[\langle \vec{h}\vec{h}^T, \Lambda^{-1} \rangle + \langle \Lambda,Q \rangle -m - \log|\Lambda| +\sum_{1 \leq k \leq r-1} \log \frac{|\Lambda_{k+1}|}{|\Lambda_k|}+ \langle \xi'(Q_1), \Lambda_1^{-1} \rangle \nonumber\\
& -& \sum_{1 \leq k \leq r-1} x_k \cdot \text{Sum}(\theta(Q_{k+1})-\theta(Q_k))]
\end{eqnarray}
where $\Lambda_r=\Lambda, \Lambda_p=\Lambda-\sum_{p \leq k \leq r-1} x_k (\xi'(Q_{k+1}) -\xi'(Q_k))$ for $1 \leq p \leq r-1.$

The limit of the free energy with self overlaps constrained to $Q$ can be expressed as
\begin{theorem}[\cite{JK1}, Theorems 1-3, Proposition 1]\label{eq:optimizers}
The limit of the free energy with self overlaps constrained to $Q$ is 
\begin{equation}\label{eq:CSFormula} 
\lim_{\epsilon \rightarrow 0} \lim_{N \rightarrow \infty} F^{\epsilon,Q}_N(\beta) = \inf_{r,\b{x},\b{Q}} \mathscr{C}_r (\b{x},\b{Q})= \inf_{x,\Phi \in \mathscr{N}_0 \times \mathscr{M}_0} \mathscr{C}_{\beta,Q} (x,\Phi)= \inf_{x,\Phi \in \mathscr{N}_0 \times \mathscr{M}_0} \mathscr{P}_{\beta,Q} (x,\Lambda,\Phi).
\end{equation}
The last two infimums are over $x(t)$ and $\Phi(t)$ defined in \eqref{eq:xpath} and \eqref{eq:ppath} such that $|Q-\Phi(t_x)|>0$ and they are both attained.
\end{theorem}

\section{The Parisi formula at zero temperature}\label{Sec3}

We first modify the setting of the Theorem \ref{eq:optimizers} to achieve convergence of the minimizer for the ground state energy. To be more specific, we claim that any matrix path minimizer of $\mathscr{C}$ lies in $ \mathscr{M}$. Moreover, we claim that any discrete path corresponds to a matrix path in $\mathscr{M}$, hence 
\begin{equation}\label{eq:discrete}
 \inf_{r,\b{x},\b{Q}} \mathscr{C}_r (\b{x},\b{Q}) \geq  \inf_{x,\Phi \in \mathscr{N_0} \times \mathscr{M}} \mathscr{C}_{\beta,Q} (x,\Phi).
 \end{equation}
 We leave the proof of this claim to next section. Moreover, as \[ \inf_{r,\b{x},\b{Q}} \mathscr{C}_r (\b{x},\b{Q}) \leq \inf_{x,\Phi \in \mathscr{N}_0 \times \mathscr{M}_0} \mathscr{C}_{\beta,Q} (x,\Phi) \text{ and }  \inf_{x,\Phi \in \mathscr{N} \times \mathscr{M}} \mathscr{C}_{\beta,Q} (x,\Phi) \geq \inf_{x,\Phi \in \mathscr{N}_0 \times \mathscr{M}_0} \mathscr{C}_{\beta,Q} (x,\Phi),
\]
we obtain that $ \inf_{r,\b{x},\b{Q}} \mathscr{C}_r (\b{x},\b{Q}) \leq \inf_{x,\Phi \in \mathscr{N} \times \mathscr{M}} \mathscr{C}_{\beta,Q} (x,\Phi)$ and therefore 
\[
\inf_{r,\b{x},\b{Q}} \mathscr{C}_r (\b{x},\b{Q}) = \inf_{x,\Phi \in \mathscr{N} \times \mathscr{M}} \mathscr{C}_{\beta,Q} (x,\Phi).
\]

Since any path $\Phi$ in $\mathscr{M}$ is continuously differentiable and $\Phi'$ is Lipschitz and uniformly bounded, then by Arzela-Ascoli theorem, for any $\beta \geq 0$, the minimizer $\Phi_{\beta,Q}$ is also a continuously differentiable function with Lipschitz and uniformly bounded derivative. Similarly, for any $\beta > 0$, we get a subsequence $\{\beta_n\}_{n \geq 0}$ such that $\{ \Phi_{\beta_n,Q} \}$ converges to a continuously differentiable path $\Phi_0$.

The following result explains the role of the minimizers of the Crisanti-Sommers formula in the case $n\geq 2$. The proof is deferred to the end of this subsection.
\begin{theorem}\label{thm:thm3}
 Assume that the pair $(x, \Phi)$ is a minimizer of the Crisanti-Sommers formula \eqref{eq:CSFormula} and write $\mu_{P}([0,q])=x(q)$. For any $F=(F_{i,j})_{1 \leq i,j \leq m}:\mathbb{R}^{m \times m} \rightarrow \mathbb{R}^{m \times m}$ continuous and bounded,
 \begin{equation*}
 \lim_{N \rightarrow \infty}\lim_{\epsilon \rightarrow 0}  \mathbb{E} \langle F(R_{1,2}) \rangle_{\epsilon, Q}= \int^m_0 F \circ \Phi(t)  d \mu_P(t).
 \end{equation*}
\end{theorem}

\begin{remark}[Uniqueness of the Parisi pair $(x, \Phi)$] \label{Rem:aloha} By choosing $F(X)=(\text{trace} (X))_{1 \leq i,j \leq m}$ and using the fact that $\text{Trace}(\Phi(t))=t$, one can see that the minimizing measure $\mu_P$ is unique. Similarly, for any $t \in \text{ supp} \mu_P$, the value of $\Phi(t)$ is also unique. However, for any $t \notin \text{ supp } \mu_P,$ we can modify $\Phi(t)$ arbitrarily as $\Phi(t)$ will not change the corresponding value of $\mathscr{C}(x,\Phi)$.
\end{remark}
 
Fix $\beta >0$. Let $(x_{\beta,Q},\Phi_{\beta,Q})$ be an optimizer of \eqref{eq:CSFormula} in $\mathscr{N}_0 \times \mathscr{M}_0$. The following lemma shows that there exists $Q_\beta \in \mathbb M$ such that $\mathscr{C}_{\beta,Q_\beta}(x_{\beta,Q_\beta},\Phi_{\beta,Q_\beta})=\sup_{Q \in \mathbb M} \mathscr{C}_{\beta,Q}(x_{\beta,Q},\Phi_{\beta,Q})$. For $Q\in \mathbb M$, let 

\[
\mathscr{C}(Q) = \inf_{{x},\Phi} \mathscr{C}_{\beta,Q}(x,\Phi) =  \mathscr{C}_{\beta,Q}(x_{\beta,Q},\Phi_{\beta,Q}).
\]

\begin{lemma}\label{lem:adasdsad}
For any $\beta>0$, the map $Q \mapsto \mathscr{C}(Q)$ is continuous. Furthermore,  there exists $Q_\beta \in \mathbb M$ such that 
\[
\mathscr{C}(Q_{\beta})=\sup_{Q\in \mathbb M} \mathscr{C}(Q).
\] 
\end{lemma}

\begin{proof}
We start by proving continuity of $\mathscr{C}$. It suffices to show that for any sequence $\{Q_n\} \in \mathbb M$ converging to $Q \in \mathbb M$, $\{ \mathscr{C}(x_n,\Phi_n) \}$ converges to $\mathscr{C}(x,\Phi)$, where $(x_n,\Phi_n)$ and $(x,\Phi)$ are minimizers of \eqref{eq:CSFormula} with constraints $Q_n$ and $Q$, respectively.

Let $F=(F_{i,j})_{1 \leq i,j \leq m}:\mathbb{R}^{m \times m} \rightarrow \mathbb{R}^{m \times m}$ be a continuous and bounded function. Since $Q_n \rightarrow Q$ as $n \rightarrow \infty$, we obtain from \eqref{cont:contin},
\begin{equation*}
\lim_{n \rightarrow \infty} \lim_{N \rightarrow \infty} \lim_{\epsilon \rightarrow 0} \mathbb{E} \langle F(R_{1,2}) \rangle_{\epsilon,Q_{n}}=\lim_{N \rightarrow \infty} \lim_{\epsilon \rightarrow 0} \mathbb{E} \langle F(R_{1,2}) \rangle_{\epsilon,Q}.
\end{equation*}

On the other hand, Theorem \ref{thm:thm3} implies
\begin{equation*}
\lim_{N \rightarrow \infty} \lim_{\epsilon \rightarrow 0} \mathbb{E} \langle F(R_{1,2}) \rangle _{\epsilon,Q}=\int^m_0 F \circ \Phi(t) d\mu(t)
\end{equation*}
and \begin{equation*}
\lim_{N \rightarrow \infty} \lim_{\epsilon \rightarrow 0} \mathbb{E} \langle F(R_{1,2}) \rangle _{\epsilon,Q_{n}}=\int^m_0 F \circ \Phi_n(t) d\mu_n(t).
\end{equation*}

Combining the above displays we obtain
\begin{equation}\label{eq:Faaa}
\lim_{n \rightarrow \infty} \int^m_0 F \circ \Phi_n(t) d\mu_n(t)=\int^m_0 F \circ \Phi(t) d\mu(t).
\end{equation}

Now fix $t_0 \in \text{supp }\mu$ and set $A_0:=\Phi(t_0)$. Applying \eqref{eq:Faaa}, with $F:\mathbb{R}^{m \times m} \rightarrow \mathbb{R}^{m\times m}$ given by 
\[
F(X)=\mathbbm{1}_{\{X=A_0\}}:=( \mathbbm{1}_{\{X=A_0\}})_{1 \leq i,j \leq m},
\] 
we obtain that
\begin{equation}\label{eq:kayak}
\lim_{n \rightarrow \infty} \int^m_0 \mathbbm{1}_{ \{ \Phi_n(t)=\Phi(t_0) \} } d \mu_n(t) = (\mu(\{ t_0 \})).
\end{equation}
Since $t_{0} \in \text{ supp } \mu$ we have Trace$(\Phi(t_0))=t_0$. Thus the indicator function above is only non-zero in a subset of $\{ t \in [0,m]: \text{Trace}(\Phi_{n}(t))=t_0 \}$. At the same time, for any $t \in \text{ supp } \mu_{n}$, Trace$(\Phi_{n}(t))=t$. These two observations, combined with \eqref{eq:kayak}, imply that for $n$ sufficiently large, $\Phi_n(t_0)=\Phi(t_0)$ and $\mu_n(\{t_0\}) \rightarrow \mu ( \{ t_0 \} )$ as $n \rightarrow \infty.$

Similarly, for any $s_0 \notin \text{supp } \mu$, consider the function $G(X)=\mathbbm{1}_{\{\text{trace }X=s_0\}} :=(\mathbbm{1}_{\{ \text{trace }X=s_{0} \}})_{1 \leq i,j \leq m} $. Another application of \eqref{eq:Faaa} implies that
\begin{equation*}
\lim_{n\to \infty}\int^m_0 \mathbbm{1}_{\{ t=s_0\}} d\mu_n(t)=0,
\end{equation*}
 which leads to  $\mu_n(\{s_0\}) \rightarrow 0$ as $n \rightarrow \infty.$ Last,  since $s_0 \notin \text{ supp } \mu$, the value of $\Phi(s_0)$ will not affect the value of $\mathscr{C}(x,\Phi)$ (see Remark \ref{Rem:aloha}). 
 
Looking back at \eqref{CSF}, the facts that $\Phi_n(t) \to \Phi(t)$ for $t \in \text{ supp } \mu$, and $\mu_n(\{s\}) \rightarrow \mu ( \{ s \} )$ $\forall s$ imply 
 $\mathscr{C}_{\beta,Q_{n}}(x_n,\Phi_n) \rightarrow \mathscr{C}_{\beta,Q} (x,\Phi)$ as $n \rightarrow \infty$ and thus continuity of $\mathscr{C}(Q)$ with respect to $Q$ in $\mathbb M$. 

The second assertion in the lemma now follows from continuity of $\mathscr{C}(Q)$ and compactness of the space $\mathbb M$. 
\end{proof}

Recall the definition of free energy with any constraint $Q\in \mathbb{S}^m_+$,
\begin{equation*}
F_N^{\epsilon,Q}(\beta)=\frac{1}{N} \mathbb{E} \log \int_{\mathcal Q_N^{\epsilon}} \exp \beta \left(H_N(\vec{\sigma})+\sum_{j=1}^m \vec{h}(j) \sum_{i=1}^N \vec{\sigma}_i (j)\right) d \lambda^m_N(\vec{\sigma}),
\end{equation*}
and denote the free energy with no constraint by
\begin{equation*}
F_N(\beta)=\frac{1}{N} \mathbb{E} \log \int_{(S_N)^m} \exp \beta \left(H_N(\vec{\sigma})+\sum_{j=1}^m \vec{h}(j) \sum_{i=1}^N \vec{\sigma}_i (j)\right) d \lambda^m_N(\vec{\sigma}).
\end{equation*}
The limiting free energy was obtained in \cite{Ko}:
\begin{theorem}[\cite{Ko}, Theorem 1]\label{thm:JK2}
For any $m \geq 1$, the limit of the free energy is give by
\begin{equation*}
\lim_{N \rightarrow \infty} F_N(\beta)=\sup_{Q \in \mathbb{M}} \inf_{x_Q,\Phi_Q,\Lambda_Q}\mathscr{P}(x_Q,\Phi_Q,\Lambda_Q,Q)=\sup_{Q \in \mathbb{M}} \inf_{x_Q,\Phi_Q}\mathscr{C}(x_Q,\Phi_Q,Q).
\end{equation*}
\end{theorem}

Let $\{ Q_\beta \}_{\beta>0}$ be a sequence given by Lemma \ref{lem:adasdsad}. Since  $\{ Q_\beta \}_{\beta>0}$ is bounded, there exists a subsequence $\{ Q_{\beta_k} \}_{k \geq 0}$ and $Q_\infty \in \mathbb M$ such that $\{ Q_{\beta_k} \}_{k \geq 0}$ converges to $Q_\infty$ as $\beta_{k} \to \infty$. Without loss of generality, we will assume $\{ Q_\beta \}$ converges to $Q_\infty.$
By Lemma \ref{lem:adasdsad}, Theorem \ref{eq:optimizers}, and Theorem \ref{thm:JK2} we have 
\begin{equation} \label{eq:freeenergy}
 \lim_{N \rightarrow \infty} F_N(\beta)=\lim_{\epsilon \rightarrow 0} \lim_{N \rightarrow \infty} F_N^{\epsilon,Q_\beta}(\beta).
\end{equation} 

Moreover, since $Q_\beta \rightarrow Q_\infty$ as $\beta \rightarrow \infty,$ then by dominated convergence theorem, we obtain
\begin{equation}\label{eq:carnaval}
\lim_{\beta \rightarrow \infty} \lim_{\epsilon \rightarrow 0} \lim_{N \rightarrow \infty} \frac{1}{\beta} F_N^{\epsilon,Q_\beta}(\beta) = \lim_{\beta \rightarrow \infty} \lim_{\epsilon \rightarrow 0} \lim_{N \rightarrow \infty} \frac{1}{\beta} F_N^{\epsilon,Q_\infty}(\beta).
\end{equation}

We now investigate the ground state energy 
\[
GSE:=\lim_{N \rightarrow \infty} \max_{\vec{\sigma} \in (S_N)^m} \frac{H_N(\vec{\sigma})}{N}.
\]

A standard computation (see \cite[Section 5]{AB}, for instance) implies that 
\begin{eqnarray*}
GSE =\lim_{\beta \rightarrow \infty}  \lim_{N \rightarrow \infty} \frac{1}{\beta} F_N(\beta) \text{ almost surely, }
\end{eqnarray*}
and using \eqref{eq:freeenergy} we obtain
\begin{eqnarray*}
GSE =\lim_{\beta \rightarrow \infty} \lim_{\epsilon \rightarrow 0} \lim_{N \rightarrow \infty} \frac{1}{\beta} F^{\epsilon,Q_\beta}_N(\beta) \text{ almost surely.}
\end{eqnarray*}
Combining with \eqref{eq:carnaval} we obtain
\begin{proposition} \label{thm:thmx} We have the following:
\begin{eqnarray*}
GSE=\lim_{N \rightarrow \infty} \max_{R(\vec{\sigma}, \vec{\sigma}) \in Q_\infty} \frac{H_N(\vec{\sigma})}{N}=\lim_{\beta \rightarrow \infty} \lim_{\epsilon \rightarrow 0} \lim_{N \rightarrow \infty} \frac{1}{\beta} F^{\epsilon,Q_\infty}_N(\beta).
\end{eqnarray*}
\end{proposition}

We finish this subsection with a proof of Theorem \ref{thm:thm3} and \eqref{eq:discrete}.

\begin{proof}[Proof of Theorem  \ref{thm:thm3}]

For each $1 \leq i \leq m$, denote the $i$-th unit vector in $\mathbb R^m$ by $\vec{e}_i=(0,\cdots,0,1,0,\cdots,0)$. Given a vector $\vec{\alpha}$, let $\vec{\alpha} \oplus_i \vec{\alpha}$ as $\vec{\alpha}\otimes \vec{e}_i+\vec{e}_i \otimes \vec{\alpha}.$ In this proof we drop from our notation the dependencies on $\beta$ and $Q$. 

For any $p \geq 2$, we consider arbitrarily $\vec{c}_p$ that also satisfies the requirement of the inverse temperature,
As the Crisanti-Sommers functional is differentiable at each $\vec{\beta}_p(i)$, where $p \geq 2, 1 \leq i \leq n$, we consider $\vec{\beta}_p+t \vec{c}_p$ as inverse temperature and compute the first derivative of the Crisanti-Sommers functional at $t=0$.

 \begin{align}\label{eq:gibz}
\frac{d\mathscr{C}}{dt}\bigg|_{t=0}&= \frac{1}{2} p \int^m_0 x(t) \langle \vec{\beta}_p \otimes \vec{c}_p+\vec{c}_p \otimes \vec{\beta}_p , \Phi(t)^{\circ (p-1)} \odot \Phi'(t) \rangle dt \nonumber \\
 &= \frac{p}{2}   \int^m_0 \int^t_0   \langle \vec{\beta}_p \otimes \vec{c}_p+\vec{c}_p \otimes \vec{\beta}_p , \Phi_{ij}(t)^{ p-1} \Phi'_{ij}(t)  \rangle d \mu_P(s) dt\nonumber \\
 &= \frac{p}{2}  \int^m_0 \int^m_s   \langle \vec{\beta}_p \otimes \vec{c}_p+\vec{c}_p \otimes \vec{\beta}_p , \Phi_{ij}(t)^{ p-1} \Phi'_{ij}(t)  \rangle dt d \mu_P(s) \nonumber \\
& = \frac{1}{2}\langle \vec{\beta}_p \otimes \vec{c}_p+\vec{c}_p \otimes \vec{\beta}_p , Q^{\circ p} - \int^m_0 \Phi(s)^{\circ p} d \mu_P(s) \rangle \nonumber \\
& = \langle \vec{\beta}_p \otimes \vec{c}_p , Q^{\circ p} - \int^m_0 \Phi(s)^{\circ p} d \mu_P(s) \rangle.
\end{align}

On the hand, integration by parts implies
\begin{equation*}
\mathbb{E}\bigg[{\frac{\partial H_{N}(\vec{\sigma})}{\partial t}} H_N(\vec{\sigma})\bigg]=N \langle \vec{c}_p \otimes \vec{\beta}_p, R_{1,2} \rangle
\end{equation*}
and
 \begin{equation}\label{eq:doug}
 \mathbb{E} \left \langle \frac{\partial H_{N}(\vec{\sigma})}{\partial t}  \right \rangle_{\epsilon,Q} = N \langle \vec{c}_p \otimes \vec{\beta}_p, Q^{\circ p} - \mathbb{E} \langle R^{\circ p}_{1,2} \rangle_{\epsilon,Q} \rangle.
 \end{equation}
Last, we note that the differential of $F^{\epsilon}_{N}(Q)$ at $\beta_p(i)$ is given by $\frac{1}{N} \mathbb{E} \langle \frac{\partial H_{N}(\vec{\sigma})}{\partial t}  \rangle_{\epsilon,Q}$, 
 so by \eqref{eq:doug}
 \begin{equation} \label{eq:diff1}
 \frac{d}{d t }F^{\epsilon}_{N}(Q)= \langle \vec{c}_p \otimes \vec{\beta}_p, Q^{\circ p} - \mathbb{E} \langle R^{\circ p}_{1,2} \rangle_{\epsilon,Q} \rangle.
 \end{equation}

Moreover, by H\"older's inequality, $F^{\epsilon}_{N}$ is convex at $\beta_p(i)$, and this combined with the fact that $\mathscr{C}$ is both convex and differentiable at $\beta_p(i)$, we get
\begin{equation*}\label{eq:wf}
\lim_{N \rightarrow \infty} \lim_{\epsilon \rightarrow 0} \frac{\partial F^{\epsilon}_N} {\partial t} = \frac {\partial \mathscr{C}}{\partial t}.
\end{equation*}

A combination of \eqref{eq:CSFormula},   \eqref{eq:gibz}, and \eqref{eq:diff1}  yields
 \begin{equation*}
\lim_{N \rightarrow \infty}\lim_{\epsilon \rightarrow 0}\langle \vec{c}_p \otimes \vec{\beta}_p, Q^{\circ p} - \mathbb{E} \langle R^{\circ p}_{1,2} \rangle_{\epsilon,Q} \rangle=\langle \vec{\beta}_p \otimes \vec{c}_p , Q^{\circ p} - \int^m_0 \Phi(s)^{\circ p} d \mu_P(s) \rangle
 \end{equation*}
which is equivalent to say that, for each $1 \leq i \leq m$,
 \begin{equation*}\label{eq:WF}
\langle \vec{\beta}_p \otimes \vec{c}_p ,  \lim_{N \rightarrow \infty}\lim_{\epsilon \rightarrow 0} \mathbb{E} \langle R^{\circ p}_{1,2} \rangle_{\epsilon,Q}  - \int^m_0 \Phi(s)^{\circ p} d \mu_P(s) \rangle=0.
 \end{equation*}
 
 As $c_p$ is chosen arbitrarily, we get the following relation,
  \begin{equation*}
 \lim_{N \rightarrow \infty}\lim_{\epsilon \rightarrow 0}  \mathbb{E} \langle R_{1,2}^{\circ p} \rangle_{\epsilon,Q} = \int^m_0 \Phi(t)^{\circ p}  d \mu_P(t).
 \end{equation*}
 Since the even polynomials are dense on $C[0,1]$, we get the desired conclusion.
\end{proof}

\begin{proof}[Proof of \eqref{eq:discrete}]

It suffices to show that each discrete path correponds to a matrix path in $\mathscr{M}$.
Consider a discrete monotone matrix path encoded by an increasing sequence of real numbers and a monotone sequence of $n \times n$ symmetric positive semidefinite matrices,\begin{eqnarray*}
0=x_0 \leq x_1 \leq \ldots \leq x_{r-2} \leq x_{r-1} \leq 1, \\
0=Q_0 \leq Q_1 \leq \ldots \leq Q_{r-2} \leq Q_{r-1} \leq Q_r =Q,
\end{eqnarray*}
where $r \geq 1$. As before, we denote $\b{x}=(x_k)^{r-1}_{k=0}$ and $\b{Q}=(Q_k)^r_{k=1}$. 

Taking $t_k:= \text{trace }(Q_k)$ we define a Lipschitz path $\Phi$ by taking $\Phi(t_k) = Q_k$ at each point $t_k$ and interpolate by sine functions:
\begin{align*}
\Phi(t_k)&=Q_k, \\
\Phi(t)&= \frac{1}{2} [ \Phi(t_k)+\Phi(t_{k+1}) ]+ \frac{1}{2} [\Phi(t_{k+1}) - \Phi(t_k)] \cdot \sin\left(\pi \frac{2t-t_k-t_{k+1}}{2(t_{k+1}-t_k)}\right) \text{ for } t_k \leq t \leq t_{k+1}.
\end{align*}
Thus,
\begin{eqnarray*}
\Phi'(t)= \frac{\pi}{2} \cdot \frac{\Phi(t_{k+1}) - \Phi(t_k)}{t_{k+1}-t_k} \cdot \cos\left (\pi \frac{2t-t_k-t_{k+1}}{2(t_{k+1}-t_k)}\right)
\end{eqnarray*}
and we set $x(t) = x_k$ for $t_k \leq t < t_{k+1}$.

It is not difficult to show from \eqref{eq:discreteCS}, that the Crisanti-Sommers functionals agree. From the construction above, we obtain that the desired $\Phi$ lies in $\mathscr{M}$, which means $\Phi$ is a continuously differentiable Lipschitz monotone matrix path with Lipschitz derivatives that satisfies \eqref{eq:ppath} satisfying the assumption that for any $t \in$ supp $\mu$, Trace $\Phi(t)=t$.
\end{proof}
\begin{remark}[Infinite differentiability of a minimizer $\Phi$] Note that in the proof above, we can also interpolate the discrete monotone matrix path linearly and then add a mollifier to smoothen the path in order to make it to be infinitely differentiable. Moreover, as the mollifier is uniformly bounded at any order of derivative, we can then conclude the infinite differentiability of a minimizer of $\mathscr{C}.$ 

\end{remark}

\subsection{Crisanti-Sommers formula of the Ground State Energy}\label{sec32}
We now turn our attention to the main result of this section, a Parisi type formula for $GSE$. Let 
\[
GSE(Q)= \lim_{\epsilon \to 0}\lim_{N \rightarrow \infty} \max_{R(\vec{\sigma}, \vec{\sigma}) \in  \mathcal Q_N^{\epsilon}} \frac{H_N(\vec{\sigma})}{N}
\]
and
\begin{equation*}
\mathscr{K}(Q):=\left \{(L,\alpha,\Phi) \in \mathbb{S}_+^m \times \mathscr{N} \times \mathscr{M}: L> \int^m_0 \alpha(s) \Phi'(s) ds \text{ and } \Phi \text{ constrained on } Q \right  \}.
\end{equation*}
For any $(L,\alpha,\Phi) \in \mathscr{K}(Q)$, define
\begin{eqnarray*}
\mathscr{C} (L,\alpha,\Phi)&=&\frac{1}{2}\bigg[ \langle \xi'(Q)+\vec{h}\vec{h}^T,L \rangle + \int^m_0 \langle (L- \int^t_0 \alpha(s) \Phi'(s)ds)^{-1}, \Phi'(t) \rangle dt \nonumber \\
&-& \int^m_0 \langle \xi''(\Phi(t)) \odot \Phi'(t),\int^t_0 \alpha(s) \Phi'(s)ds \rangle dt  \bigg].
\end{eqnarray*}

Set

\begin{theorem}(Parisi's formula for the ground state energy.) \label{thm:thmx0} For any vector mixed $p$-spin model and any constraint $Q$ we have
\begin{equation*}
GSE(Q)=\inf_{(L,\alpha, \Phi) \in \mathscr{K}(Q)} \mathscr{C}(L,\alpha,\Phi).
\end{equation*}
Moreover, 
\begin{equation}\label{eq:GSEW}
GSE=\inf_{(L,\alpha, \Phi) \in \mathscr{K}(Q_{\infty})} \mathscr{C}(L,\alpha,\Phi),
\end{equation}
and the minimizers $(L_{0},\alpha_{0},\Phi_{0})$ of \eqref{eq:GSEW} satisfy
\begin{equation*}
\alpha_0:= \lim_{\beta \rightarrow \infty} \beta x_{\beta,Q_\infty} \text{ vaguely on } [0,m),
\end{equation*}
\begin{equation*}
\Phi_0:= \lim_{\beta \rightarrow \infty} \Phi_{\beta,Q_\infty} \text{ uniformly and  } \Phi'_0:= \lim_{\beta \rightarrow \infty} \Phi'_{\beta,Q_\infty} \text{ uniformly,}
\end{equation*}
\begin{equation*}
L_0:= \lim_{\beta \rightarrow \infty} \int^m_0 \beta x_{\beta,Q_\infty}(s) \Phi_{\beta,Q_\infty}'(s) ds. 
\end{equation*}\end{theorem}

\begin{remark}
The vague convergence of $(\beta \Phi_{\beta,Q_\infty})_{(\beta>0)}$ on $[0,m)$ means that $\lim_{\beta \rightarrow \infty} \beta x_{\beta,Q_\infty} (s) = \alpha_0(s)$ at all points of continuity of $\alpha_0$ on $[0,m)$. 
\end{remark}

\subsubsection{Example: Multi-dimensional SK model}
Consider the multi-dimensional SK model, $i.e.$ for $A \in \mathbb{S}^m_+, \xi(A)=(\vec{\beta}_2 \otimes \vec{\beta}_2) \circ A^{\circ 2}$. Formula \eqref{eq:GSEW} can be explicitly solved and we find that the multi-dimensional SK model is replica symmetric at zero temperature.
\begin{proposition} 
The multi-dimensional SK model is replica symmetric at zero temperature, that is, 
the minimizer $(L_0,\alpha_0,\Phi_0)$ is given by 
\begin{eqnarray} \label{minimizer0}
L_0=Q^{\frac{1}{2}}(Q^{\frac{1}{2}} (\xi'(Q)+\vec{h}\vec{h}^T) Q^{\frac{1}{2}})^{-\frac{1}{2}}Q^{\frac{1}{2}}, \quad \alpha_0=0 \quad \text{ and } \Phi_0=\frac{t}{m}Q
\end{eqnarray}
and the corresponding GSE is equal to
\begin{equation*}
\text{ Sum } ((Q^{\frac{1}{2}}(\xi'(Q)+\vec{h}\vec{h}^T) Q^\frac{1}{2})^{\frac{1}{2}}).
\end{equation*}
\end{proposition}

We defer the proof of this Proposition to Section \ref{Sec5}.

\subsubsection{Proof of Theorem \ref{thm:thmx0}}

The rest of the section covers the proof of Theorem \ref{thm:thmx0}.

Before the proof of Theorem \ref{thm:thmx0}, we need to first introduce some notation. For any matrix $A$ and $\vec{x}_p$ satisfying the requirements in \eqref{Hamiltonian} , we denote the corresponding Hamiltonian by 
\begin{equation*}
X^j_N(\vec{\sigma})=\sum_{p\geq 2} \vec{x}_p(j) H_{N,p}(\vec{\sigma}(j)).
\end{equation*}
and the covariance by $\zeta(A)=\Sigma_{p \geq 2} (\vec{x}_p \otimes \vec{x}_p) \odot A^{\circ p}$.

\begin{lemma}\label{thm:thmx1}
There exists a constant $C_{\zeta}$ depending only on $\zeta$ such that for any $\beta>0$, 
\begin{equation}\label{25}
\beta x_{\beta,Q}(q) \leq \frac{C_{\zeta}}{Sum(\zeta(Q)-\zeta(\Phi_{\beta,Q}(q)))}, \forall q \in [0,m].
\end{equation}
\end{lemma}

\begin{proof}[Proof of Lemma \ref{thm:thmx1}]

Note that for any $N \geq 1$, by Dudley's entropy integral, 
\begin{equation}\label{a1}
\mathbb{E} \max_{\vec{\sigma} \in S^m_N} \frac{X_N(\vec{\sigma})}{N} \leq C_{\zeta}.
\end{equation}Here the constant $C_{\zeta}>0$ depends only on $\zeta$.

From Gaussian integration by parts, we obtain,
\begin{equation}\label{a2}
\beta(\zeta(Q)-\mathbb{E} \langle \zeta(R(\vec{\sigma}^1,\vec{\sigma}^2)) \rangle_\beta)=\mathbb{E} \langle \frac{X_N(\vec{\sigma})}{N} \rangle_\beta
\end{equation}
where $\langle \cdot \rangle_\beta$ is the Gibbs average with respect to the Gibbs measure $G_{N,\beta}(\sigma)$ defined by
\begin{equation*}
G_{N,\beta}(\vec{\sigma})=\frac{\exp \beta X_N(\vec{\sigma})}{Z_N(\beta)}.
\end{equation*}
From the differentiability of $\vec{\beta}_p$, we also have
\begin{equation}\label{a3}
\lim_{N \rightarrow \infty} \mathbb{E} \langle \zeta(R(\vec{\sigma}^1, \vec{\sigma}^2)) \rangle_\beta = \int^m_0 \zeta(\Phi_{\beta,Q}(s)) x_{\beta,Q} (ds).
\end{equation}
By \eqref{a1}, \eqref{a2} and \eqref{a3}, we then obtain, 
\begin{equation*}
\beta Sum(\zeta(Q)-\int^m_0 \zeta(\Phi_{\beta,Q}(s)) x_{\beta,Q} (ds)) 
=\mathbb{E} \langle \frac{X_N(\vec{\sigma})}{N} \rangle_\beta
\leq \mathbb{E} \max_{\sigma \in S_N} \frac{X_N(\sigma)}{N} \leq C_{\zeta}.
\end{equation*}
Finally combining with the following two inequalities which can be derived from integration by parts,
\begin{equation*}\label{28}
\int^m_0 \beta x_{\beta,Q}(s) {\zeta} (\Phi_{\beta,Q}(s))' \odot \Phi_{\beta,Q}'(s)  ds =\beta (\zeta(Q)-\int^m_0 \zeta(\Phi_{\beta,Q}(s)) x_{\beta,Q} (ds))
\end{equation*}
and
\begin{equation*}
\int^m_q \beta x_{\beta,Q}(s) \zeta(\Phi(s))' \odot \Phi_{\beta,Q}'(s) ds \geq \beta x_{\beta,Q}(q) (\zeta(Q)-\zeta(\Phi_{\beta,Q}(q))), \forall q \in [0,m]
\end{equation*}
we then obtain that,
\begin{equation*}
 \beta x_{\beta,Q}(q) Sum(\zeta(Q)-\zeta(\Phi_{\beta,Q}(q))) \leq C_\zeta , \forall q \in [0,m],
 \end{equation*}
as desired.
\end{proof}

\begin{lemma}\label{thm:thmx2}
There exists a constant $C'_\xi >0$ and a positive semidefinite matrix $A_\xi$ depending only on $\xi$ such that 
\begin{equation*}\label{lemma47}
\limsup_{\beta \rightarrow \infty} \beta (m-q_\beta) \leq C'_\xi
\end{equation*}
and 
\begin{eqnarray*}
\limsup_{\beta \rightarrow \infty} \int^m_0 \beta x_{\beta,Q}(s) \Phi'_{i,j}(s) ds \leq A_\xi.
\end{eqnarray*}
\end{lemma}
\begin{proof}[Proof of Lemma \ref{thm:thmx2}]
From Lemma \ref{thm:thmx1}, we have 
\begin{equation}\label{eq:fdsfrt}
\beta = \beta x_{\beta,Q}(q_\beta) \leq \frac{C_{\zeta}}{Sum(\zeta(Q)-\zeta(\Phi_{\beta,Q}(q_\beta)))}.
\end{equation}
Thus the denominator on the right-side of \eqref{eq:fdsfrt} must go to $0$ and we obtain
\begin{equation}\label{qbeta}
\lim_{\beta \rightarrow \infty} q_\beta =m.
\end{equation}

On the other hand, as $Sum(\zeta)$ is non-decreasing, by the mean value theorem,
\begin{equation*}
\beta Sum( \zeta(\Phi_{\beta,Q}(q_\beta))' \odot (Q-\Phi_{\beta,Q}(q_\beta))) \leq \beta Sum(\zeta(Q)-\zeta(\Phi_{\beta,Q}(q_\beta))) \leq C_{\xi}.
\end{equation*}
Consequently,
\begin{equation*}
  Sum( \zeta'(Q)\odot \limsup_{\beta \rightarrow \infty} \beta (Q-\Phi_{\beta,Q}(q_\beta)) )\leq  C_{\zeta}.
\end{equation*}
By the arbitrariness of $\zeta$, the inequality above implies that $\limsup_{\beta \rightarrow \infty} \beta (Q-\Phi_{\beta,Q}(q_\beta))$ is a bounded matrix. Moreover, by integration by parts, we obtain,
\begin{eqnarray*}
C_{\zeta}  \geq  \int^m_0 \beta x_{\beta,Q}(s) \zeta' (\Phi_{\beta,Q}(s)) \odot \Phi_{\beta,Q}'(s)  ds \geq \int^m_{\frac{m}{2}} \beta x_{\beta,Q}(s) \zeta' (\Phi_{\beta,Q}(s)) \odot \Phi_{\beta,Q}'(s)  ds\\
\geq \zeta' (\Phi_{\beta,Q}(\frac{m}{2})) \odot \int^m_{\frac{m}{2}} \beta x_{\beta,Q}(s)  \Phi_{\beta,Q}'(s)  ds.
\end{eqnarray*}
Similarly, by the arbitrariness of $\zeta$, the inequality above implies that $\limsup_{\beta \rightarrow \infty} \int^m_{\frac{m}{2}} \beta x_{\beta,Q}(s)  \Phi_{\beta,Q}'(s)  ds$ is a bounded matrix.

Finally,
\begin{eqnarray*}
\int^m_0 \beta x_{\beta,Q}(s) \Phi'_{\beta,Q}(s) ds = \int^{\frac{m}{2}}_0 \beta x_{\beta,Q}(s)  \Phi_{\beta,Q}'(s)  ds+\int^m_{\frac{m}{2}} \beta x_{\beta,Q}(s)  \Phi_{\beta,Q}'(s)  ds \\
\leq \int^{\frac{m}{2}}_0 \frac{C_{\xi}}{Sum(\xi(Q)-\xi(\Phi_{\beta,Q}(q)))} dq + \int^m_{\frac{m}{2}} \beta x_{\beta,Q}(s)  \Phi_{\beta,Q}'(s)  ds \\
\leq  {\frac{m}{2}} \frac{C_{\xi}}{Sum(\xi(Q)-\xi(\Phi_{\beta,Q}(\frac{m}{2})))}  + \int^m_{\frac{m}{2}} \beta x_{\beta,Q}(s)  \Phi_{\beta,Q}'(s)  ds.
\end{eqnarray*}
As $\lim_{\beta \rightarrow \infty} \Phi_{\beta,Q}(t) = \Phi_0(t)$ converges uniformly, we obtain that for $\beta$ sufficiently large, $\int^m_0 \beta x_{\beta,Q}(s) \Phi'_{\beta,Q}(s) ds$ is a uniformly bounded matrix.
\end{proof}

Combining the lemmas above, we can use the Helly's selection theorem combined with a diagonalization process to guarantee the vague convergence of $(\beta_n x_{\beta_n,Q})_{n \geq 1}$. As for  $\Phi_0(t)$ satisfying \eqref{eq:xpath}, we can guarantee the existence of the following limit :
\begin{eqnarray*}
\lim_{k \rightarrow \infty} \beta_{n_k} x_{\beta_{n_k},Q} \in \mathscr{N} \text{ vaguely on } [0,m) \\
\lim_{k \rightarrow \infty} \beta_{n_k} (Q-\Phi_{\beta_{n_k},Q}(q_{\beta_{n_k}})),\\
\lim_{k \rightarrow \infty} \int^m_0 \beta_{n_k}x_{\beta_{n_k},Q}(s) \Phi_{\beta_{n_k},Q}'(s)ds.
\end{eqnarray*}
Without loss of generality, we can assume all these convergences hold for the sequence $(\beta_n)_{n \geq 1}$ and denote
\begin{align}\label{eq:zerocons}
\alpha_0 &:= \lim_{n \rightarrow \infty} \beta_{n} x_{\beta_n,Q} \in \mathscr{N} \text{ vaguely on } [0,m) \nonumber \\
\Omega_0 &:= \lim_{n \rightarrow \infty} \beta_{n} (Q-\Phi_{\beta_n,Q} (q_{\beta_{n}})),\\
L_0 &:= \lim_{n \rightarrow \infty} \int^m_0 \beta_{n}x_{\beta_{n},Q}(s) \Phi_{\beta_n,Q}'(s) ds.\nonumber
\end{align}

\begin{lemma}\label{thm:thmx3} Let $(x_{\beta,Q},\Lambda_{\beta,Q},\Phi_{\beta,Q})$ be a minimizer of \eqref{eq:contPAr} and $\mu_{\beta}([0,t]) = x_{\beta,Q}(t)$. 
For any q in the support of $\mu_\beta$, we have that 
\begin{align} \label{41}
\Phi_{\beta,Q}(q)&=(\Lambda_{\beta,Q}-D^{x_{\beta,Q}}_\beta(0))^{-1} \vec{h}\vec{h}^T (\Lambda_{\beta,Q}-D^{x_{\beta,Q}}_\beta(0))^{-1}
\nonumber
\\&+ \int^q_0 (\Lambda_{\beta,Q}-D^{x_{\beta,Q}}_\beta(s))^{-1}(\xi''_\beta(\Phi_{\beta,Q}(s))\odot \Phi_{\beta,Q}'(s))(\Lambda_{\beta,Q}-D^{x_{\beta,Q}}_\beta(s))^{-1} ds
\end{align} 
where $D^{x_{\beta,Q}}_\beta(q)=\int^m_q x(s) \xi''_\beta(\Phi(s)) \odot \Phi_{\beta,Q}'(s) ds$.

\end{lemma}

\begin{proof}[Proof of Lemma \ref{thm:thmx3}]

Consider arbitrarily probability measure $\mu$ with distribution function $y$, and set $z(t):= (1-\epsilon) x_{\beta,Q}(t)+\epsilon y(t)$. Then $y(t)-x_{\beta,Q}(t)=\int^t_0 d(\mu-\mu_\beta)(s)$. By a straight-forward computation, we obtain that, 
\begin{equation*}
\partial_\epsilon \mathscr{P}_\beta(\Lambda_{\beta,Q},z, \Phi_{\beta,Q})\bigg|_{\epsilon=0}= \int ^{\hat{T}}_0 (y(t)-x_{\beta,Q}(t)) \left \langle \Gamma(t), \xi''(\Phi_{\beta,Q}(t)) \odot \Phi_{\beta,Q}'(t) \right \rangle dt \geq 0,
\end{equation*}
where 
\begin{align*}
\Gamma(t)&=(\Lambda_{\beta,Q}-D^{x_{\beta,Q}}_\beta(0))^{-1} \vec{h}\vec{h}^T (\Lambda_{\beta,Q}-D^{x_{\beta,Q}}_\beta(0))^{-1}\\&+ \int^t_0 (\Lambda_{\beta,Q}-D^{x_{\beta,Q}}_\beta(s))^{-1}(\xi''_\beta(\Phi_{\beta,Q}(s))\odot \Phi_{\beta,Q}'(s))(\Lambda_{\beta,Q}-D^{x_{\beta,Q}}_\beta(s))^{-1} ds-\Phi_{\beta,Q}(t).
\end{align*}

By Fubini's theorem,
\begin{equation*}
\int^m_0 \int^m_s \left \langle \Gamma(t), \xi''(\Phi_{\beta,Q}(t)) \odot \Phi_{\beta,Q}'(t) \right \rangle dt d \mu(s) \geq \int^m_0 \int^m_s \left \langle \Gamma(t), \xi''(\Phi_{\beta,Q}(t)) \odot \Phi_{\beta,Q}'(t) \right \rangle dt d \mu_\beta(s)
\end{equation*} 
which implies that 
\begin{equation*}
\int^m_0 \bar{\Gamma}(s) d \mu(s) \geq \int^m_0  \bar{\Gamma}(s) d \mu_\beta(s)
\end{equation*} 
where $\bar{\Gamma}(s):= \int^m_s \left \langle \Gamma(t), \xi''(\Phi_{\beta,Q}(t)) \odot \Phi_{\beta,Q}'(t) \right \rangle ds$.
Since the inequality holds for all $\mu$, this is equivalent to say that 
\begin{equation*}
\bar{\Gamma}(t) \geq \int^m_0 \bar{\Gamma}(s) d \mu_\beta (s)
\end{equation*}
for all $t \in [0,m]$ and equality holds for every point in supp $\mu_\beta$. Note that if $s \in \text{ supp } \mu_\beta \cap (0,m)$, we obtain
\begin{equation*}
\frac{d}{dt}\bar{\Gamma}(t) = - \left \langle \Gamma(t), \xi''(\Phi_{\beta,Q}(t)) \odot \Phi_{\beta,Q}'(t) \right.\rangle =0.
\end{equation*}

If $t$ is an isolated point in supp $\mu_\beta$, then $\Phi'(t)$ can be any positive semidefinite matrix without changing the value of $\mathscr{C}$, which implies that $\Gamma(t)=0$ and thus \eqref{41} holds.   
Moreover, if $t$ is not an isolated point, then as $\Phi_{\beta,Q}(t)$ is a continuously differentiable matrix path, we can derive the same conclusion that $\Gamma(t)=0$ by continuity. 
\end{proof}

\begin{lemma}\label{thm:thmx4}
For any q in the support of $\mu_\beta$, we have that 
\begin{equation}\label{parisi1}
\Lambda_{\beta}-D^{x_{\beta}}_\beta(q) =\left ( \int^m_q x_{\beta,Q}(s) \Phi_{\beta,Q}'(s) ds \right)^{-1}.
\end{equation}
\end{lemma}

\begin{proof}[Proof of Lemma \ref{thm:thmx4}]

Recall the $r$ step discretization of the Parisi functional given by \eqref{eq:discreteParisi}.
For $r \geq 1$, denote by $\mathscr{M}_r$ the space of all step functions $x \in \mathscr{N}$ with at most $r$ jumps and by $\mathscr{M}_r'$ the space of all $(x,\Lambda,Q)$ with $x \in \mathscr{M}_r$ and $\Lambda \in \mathbb{S}^m_+$ satisfying $\Lambda > \int^m_0 \beta x(s) \xi''_\beta(\Phi(s)) \odot \Phi'(s) ds$.

Let $({x_r}, \Lambda^r,{Q_{r}})$  be the minimizer of $\mathscr{P}_\beta$ restricted to $\mathscr{M}_r$. 
Based on the critical point of the Parisi formula, i.e., for the minimizer of the $r$ step discretization, $\partial_{Q_p} \mathscr{P}_r=0$ and $\partial_{Q_p} \mathscr{C}_r=0$, we obtain the following equation:
\begin{equation*}\label{cp1}
\Lambda_1^{-1}(\vec{h}\vec{h}^T+\xi'(Q_1))\Lambda_1^{-1}+\Sigma_{1 \leq k \leq p-1} \frac{1}{x_k}(\Lambda_k^{-1}-\Lambda_{k+1}^{-1})=Q_p 
\end{equation*}
\begin{equation} \label{cp2}
-\vec{h}\vec{h}^T+D_1^{-1}Q_1D_1^{-1}+\sum_{1 \leq k \leq p-1} \frac{1}{x_k}(D_{k+1}^{-1}-D_k^{-1}) =\xi'(Q_p) \text{ for } 2 \leq p \leq r-1
\end{equation}
and for $p=1$,
\begin{equation*}
\Lambda_1^{-1}(\vec{h}\vec{h}^T+\xi'(Q_1)) \Lambda_1^{-1} = Q_1, \quad
\xi'(Q_1)=-\vec{h}\vec{h}^T+D_1^{-1}Q_1D_1^{-1}.
\end{equation*}
Moreover, based on the extremality over $\Lambda^k$, we obtain
\begin{equation*}
- \Lambda_1^{-1} (\vec{h} \vec{h}^T + \xi'(Q_1)) \Lambda_1^{-1} +Q-(\Lambda^k)^{-1}+\Sigma_{1 \leq p \leq r-1} \frac{1}{x_p} (\Lambda_{p+1}^{-1}-\Lambda_p^{-1})=0.
\end{equation*}

Combining with \eqref{cp1} for $p=r-1$ and set $D^{x,p}_{\beta,r}=\Sigma_{p \leq k \leq r-1} x_k (\xi'(Q_{k+1})- \xi'(Q_k)) $ to be the discrete form of $D^x_{\beta,Q}(q)$ in $r-$step discretization, then we get the following relation:
\begin{equation*}
\Lambda_{r-1}=\Lambda-D^{x,r-1}_{\beta,r}=(Q-Q_{r-1})^{-1}.
\end{equation*}

Relation \eqref{cp2} also implies that $x_k(\xi'(Q_{k+1})-\xi'(Q_k))=D^{-1}_{k+1}-D^{-1}_k$ for $1 \leq k <r-1.$ Hence based on the equations above, we get
\begin{eqnarray*}
\Lambda_{p+1}-\Lambda_p&=& x_p(\xi_\beta'(Q_{p+1})-\xi_\beta'(Q_p)) \\
&=&D^{-1}_{p+1}-D_{p}^{-1}
\end{eqnarray*}
Since $D_{r-1}=Q -Q_{r-1}=\Lambda_{r-1}^{-1}$, we then get $\Lambda_p=D^{-1}_{p}$ for $1 \leq p \leq r-1$.

Then $(x_r,\Lambda^r,Q_r)$ satisfies $\Lambda^n-\Lambda(q)=(\int^m_q x_r(s) \Phi'(s) ds)^{-1}$ for all $q$ in the support of the probability measure $\mu_r$ induced by $x_r$. By the uniqueness of the minimizer $(x_r,\Lambda^r)$, we may pass to a subsequence of $(x_r,\Lambda^r)$ such that its limit equals $(x_\beta, \Lambda_\beta,Q_\beta)$. Thus we have the desired conclusion.
\end{proof}

Recall the constants defined in \eqref{eq:zerocons}.

\begin{lemma}\label{thm:thmx5}
We have $0<\Omega_{0}<\infty$ and for any $\Phi$ satisfying \eqref{eq:xpath}, $L_{0} >\int^m_0 \alpha_0(s) \Phi_0'(s) ds$.
\end{lemma}

\begin{proof}[Proof of Lemma \ref{thm:thmx5}]

Assume by contradiction that $\Omega_0=0$.
We first notice that the equation \eqref{parisi1} for $q=q_{\beta_n}$ reads as
\begin{equation*}
\Lambda_{\beta_n}=\xi'(Q)-\xi'(\Phi_{\beta_n}(q_{\beta_n}))+( Q- \Phi_{\beta_n}(q_{\beta_n})  )^{-1},
\end{equation*}
hence
\begin{equation*}
\lim_{n \rightarrow \infty} \frac{\Lambda_{\beta_n}}{\beta_n}=\Omega_0^{-1}=\infty.
\end{equation*}
Combining with the fact that $D^{x_{\beta_n}}_{\beta_n}(q)$ is bounded for any $q \in [0,m]$,
we obtain $[{\beta_n}^{-1}(\Lambda_{\beta_n}-D^{x_{\beta_n}}_{\beta_n}(q))]^{-1} $ converges to zero uniformly on $[0,m]$.

By \eqref{41}, we can reach a contradiction as follows,
\begin{eqnarray*}
Q &=&\lim_{n \rightarrow \infty}\Phi_{\beta_n}(q_{\beta_n})\\
&=&\lim_{n \rightarrow \infty} [{\beta_n}^{-1}(\Lambda_{\beta_n}-D^{x_{\beta_n}}_{\beta_n}(0))]^{-1} \vec{h}\vec{h}^T [{\beta_n}^{-1} (\Lambda_{\beta_n}-D^{x_{\beta_n}}_{\beta_n}(0))]^{-1}\\
&&+ \int^{q_{\beta_n}}_0 [{\beta_n}^{-1} (\Lambda_{\beta_n}-D^{x_{\beta_n}}_{\beta_n}(s))]^{-1}(\xi''_{\beta_n}(\Phi_{\beta_n,Q}(s))\odot \Phi_{\beta_n,Q}'(s))[{\beta_n}^{-1} (\Lambda_{\beta_n}-D^{x_{\beta_n}}_{\beta_n}(s))]^{-1} ds.
\end{eqnarray*} 
In conclusion, $0 < \Omega_0.$ A similar argument shows that $\Omega_{0}<\infty$.

Next, note that from \eqref{qbeta},
\begin{equation*}
\lim_{k \rightarrow \infty} q_{\beta_k} =m.
\end{equation*}
Therefore, for any fixed $q \in (0,m)$
\begin{eqnarray*}
\int^m_0 \beta_n x_{\beta_n,Q}(s) \Phi_{\beta_n,Q}'(s) ds &=& \int^{q_{\beta_n}}_0 \beta_n x_{\beta_n,Q}(s) \Phi_{\beta_n,Q}'(s)ds + \int ^m_{q_{\beta_n}} x_{\beta_n,Q}(s) \Phi_{\beta_n,Q}'(s) ds\\
&\geq&  \int^{q_{\beta_n}}_0 \beta_n x_{\beta_n}(s) \Phi'_0(s)ds+\beta_n (Q-\Phi_{\beta_n}(q_{\beta_n}))
\end{eqnarray*}
Then by the dominated convergence theorem, we obtain
\begin{equation*}
L_0 \geq \int ^q_0 \alpha_0(s) \Phi_0'(s) ds + \Omega_0.
\end{equation*}
Since this holds for all $q \in (0,m)$, by let $q$ tend to $m$, we obtain,
\begin{equation*}
L_0 \geq \int ^m_0 \alpha_0(s) \Phi_0'(s) ds + \Omega_0.
\end{equation*}
As $\Omega_0>0$, we get the desired conclusion.

\end{proof}

Now let us rewrite the Crisanti-Sommers functional \eqref{CSF} as follows:
\begin{eqnarray*}\label{NCSF}
 \mathscr{C}_{\beta,Q}(x,\Phi) &=& \frac{1}{2}\bigg[\langle \xi'_\beta(\Phi(m))+\vec{h}_\beta \vec{h}_\beta^T, \check{x}(m) \rangle -\int^m_0 \langle \check{\Phi}(t), \xi_\beta''(\Phi(t)) \odot \Phi'(t) \rangle dt  \nonumber \\
 &+& \int^{t_x}_0 \langle (\check{\Phi}(m)-\check{\Phi}(t))^{-1},\Phi'(t) \rangle dt + \log | \Phi(m)-\Phi(t_x)|\bigg]
\end{eqnarray*}
where $\check{\Phi}(t):= \int ^t_0 x(s) \Phi'(s) ds$.

\begin{proof}[Proof of Theorem \ref{thm:thmx0}] We will consider that case $Q=Q_{\infty}$ as the general case is similar. We set $(x_\beta,\Phi_\beta)$ as the minimizer under the constraint $Q=Q_\infty$. Now We start with the lower bound. 

For any fixed $\Phi$ satisfying \eqref{eq:xpath}, based on the definition of $L_0$ and $\alpha_0$, we obtain that 
\begin{eqnarray*}
L_0 = \lim_{\beta \rightarrow \infty} \int^m_0 \beta_n x_{\beta_n}(s) \Phi_{\beta_n}'(s) ds = \lim_{n \rightarrow \infty} \beta_n \check{\Phi}_{\beta_n}(m) \\
\lim_{n \rightarrow \infty} \beta_n \check{\Phi}_{\beta_n}(q)=\int^q_0 \alpha_0(s) \Phi'_0(s)ds, q \in [0,m).
\end{eqnarray*}
In addition, from Lemma \ref{thm:thmx5},
\begin{equation*}
\lim_{n \rightarrow \infty} \frac{\log | Q-\Phi_{\beta_n}(q_{\beta_n}) |}{\beta_n} = \lim_{n \rightarrow \infty} \frac{\log | \beta_n(Q-\Phi_{\beta_n}(q_{\beta_n})) |}{\beta_n}-m\frac{\log \beta_n}{\beta_n}=0.
\end{equation*}
Hence by applying  Fatou's lemma and the bounded convergence theorem, we obtain
\begin{eqnarray*}
GSE &=&  \lim_{n \rightarrow \infty} \frac{\mathscr{C}_{\beta_n} (x_{\beta_n},\Phi_{\beta_n})}{\beta_n} \\
        &\geq& \frac{1}{2}  \bigg[\langle \xi'(Q)+\vec{h}\vec{h}^T, \lim_{n \rightarrow \infty} \beta_n \check{\Phi}_{\beta_n}(m) \rangle -\int^m_0 \langle  \lim_{n \rightarrow \infty} \xi''(\Phi_{\beta_n}(t)) \odot \Phi_{\beta_n}'(t)  ,\lim_{n \rightarrow \infty} \beta_n \check{\Phi}_{\beta_n}(t)  \rangle dt  \nonumber \\
 &+& \int^{m}_0 \lim_{n \rightarrow \infty} \mathbbm{1}_{[0,q_{\beta_n}]} \langle (\beta_n (\check{\Phi}_{\beta_n}(m)-\check{\Phi}_{\beta_n}(t))^{-1},\Phi_{\beta_n}'(t) \rangle dt + \lim_{n \rightarrow \infty} \frac{\log | Q-\Phi_{\beta_n}(q_{\beta_n})|}{\beta_n} \bigg]\\
&=&  \mathscr{C}(L_0, \alpha_0, \Phi_0)
 \end{eqnarray*}
 which implies that 
 \begin{equation*}
 GSE \geq \inf_{(L,\alpha,\Phi) \in \mathscr{K}(Q)} \mathscr{C} (L,\alpha,\Phi).
 \end{equation*}

We now move the matching upper bound for $GSE$.
Consider any $(L,\alpha,\Phi) \in \mathscr{K}(Q)$. As $\alpha \in \mathscr{N}_0$,  there exists $c>0$ such that $\alpha(t)$ is constant for all $t \in [m-c,m]$. Then for any $[m-c_,m]$, we can set $\Phi(t)=\frac{t-m+c}{c}Q+\frac{m-t}{c}\Phi(m-c)$ without changing the value of $\mathscr{C}(L,\alpha,\Phi)$. 

As $Q- \Phi(m-c)$ is positive definite, then $\Phi(m-\frac{c}{2})=\frac{1}{2}(Q+\Phi(m-c))$ is positive definite. There exists $\epsilon$ sufficiently small such that $\frac{1}{2}(Q-\Phi(m-c))-\epsilon \Omega_1$ is still positive definite, where $\Omega=L- \int^m_0 \alpha(s) \Phi'(s) ds >0$ and $\Omega_1=\frac{\Omega}{\text{Trace}(\Omega)}$.

Take $v_\beta=m-\frac{1}{\beta} \text{Trace}(\Omega)$ and without loss of generality we can assume $v_\beta>m-\epsilon$ by choosing $\beta$ sufficiently large. Consider $y_\beta \in \mathscr{M}$ defined by $y_\beta(s)=1$ on $[v_\beta,m]$ and $y_\beta(s)=\min (\alpha(s)/\beta,1)$ on $[0,v_\beta)$.

Next, for each $v_\beta$, we customize $\Phi$ by $\Phi_{\beta}$ as follows,
 \begin{equation*} \Phi_{\beta}(t)= \left\{
\begin{array}{lcl}
\Phi(t) \qquad  \text{   for   } t \in [0,m-c] ,   \\
Q-(m-v_\beta) \cdot \Omega_1  \qquad \text{ for } t=v_\beta, \\
\frac{v_\beta-t}{v_\beta-m+c}\Phi(m-c)+\frac{t-m+c}{v_\beta-m+c}\Phi(v_\beta) \qquad \text{ for } t \in [m-c,v_\beta] \\
Q-(m-t) \cdot \Omega_1 \qquad \text{ for } t \in [v_\beta,m]
\end{array} \right. \end{equation*} 
Notice that $\lim_{\beta \rightarrow \infty} \Phi_{\beta}(t)=\Phi(t)$ and $\lim_{\beta \rightarrow \infty} \Phi'_\beta(t)=\Phi'(t)$ for $t \in [0,m].$

Based on the definition $y_\beta$, we notice that $(\beta y_\beta)$ converges vaguely to $\alpha$ on $[0,m)$ with
\begin{equation*}
\int^m_0 \beta y_\beta(s) \Phi'_\beta(s)ds = \beta(\Phi_\beta(m)-\Phi_\beta(v_\beta))+\int^{v_\beta}_0 \min(\alpha(s),\beta)\Phi_{\beta}'(s)ds \rightarrow \Omega + \int^m_0 \alpha(s) \Phi'(s) ds =L.
\end{equation*}

We claim that 
\begin{equation}\label{claim}
\lim_{\beta \rightarrow \infty} \frac{\mathscr{C}_{\beta,Q} (y_\beta,\Phi_{\beta})}{\beta} =\mathscr{C}(L,\alpha,\Phi).
\end{equation} 
If \eqref{claim} is valid, from the Crisanti-Sommers formula, we have
\begin{equation*}
GSE=\lim_{\beta \rightarrow \infty} \lim_{\epsilon \rightarrow 0} \lim_{N \rightarrow \infty} \frac{F^{\epsilon,Q_{\infty}}_N(\beta)}{\beta} = \lim_{\beta \rightarrow \infty} \inf_{x,\Phi} \frac{\mathscr{C}_{\beta,Q_{\infty}} (x,\Phi)}{\beta} \leq \lim_{\beta \rightarrow \infty} \frac{\mathscr{C}_{\beta,Q_{\infty}} (y_\beta,\Phi_{\beta})}{\beta} =\mathscr{C}(L,\alpha,\Phi).
\end{equation*}

Since this is true for any $\alpha$ and $L> \int^m_0 \alpha(s) \Phi'(s) ds$,
we get that 
\begin{equation*}
GSE \leq \inf_{(L,\alpha,\Phi) \in \mathscr{K}(Q)} \mathscr{C}(L,\alpha,\Phi).
\end{equation*}
Now it suffices for us to prove the claim \eqref{claim}.

Firstly, 
\begin{equation*}
\frac{\log | \Phi_{\beta}(m)- \Phi_{\beta}(v_\beta) |}{\beta} = \frac{  \log|\beta( \Phi_{\beta}(m)- \Phi_{\beta}(v_\beta) ) |}{\beta} - m \frac{ \log \beta}{\beta} \rightarrow 0.
\end{equation*}

Secondly,
\begin{eqnarray*}
\frac{1}{\beta}  \langle \xi_\beta'(Q)+\vec{h}_\beta \vec{h}_\beta^T,  \check{\Phi}_\beta(m) \rangle 
&=&  \langle \xi'(Q)+\vec{h} \vec{h}^T, \int^m_0 \beta y_\beta(t) \Phi_{\beta}'(t)  dt \rangle \\
&\rightarrow&  \langle \xi'(Q)+\vec{h}\vec{h}^T, L \rangle.
\end{eqnarray*}

Thirdly, since $\beta \check{\Phi}_\beta (q) \leq \int^m_0 \beta y_\beta (s) \Phi_{\beta}'(s)ds$ for $q \in [0,m]$ and $\int^m_0 \alpha(s) \Phi_{\beta}'(s) ds < L$, by bounded convergence theorem,
\begin{align*}
\frac{1}{\beta} \int^m_0 \langle \xi''(\Phi_{\beta}(t)) \odot \Phi_{\beta}'(t) , \check{\Phi}_\beta(t) \rangle dt 
&= \int^m_0 \langle \xi''(\Phi_{\beta}(t)) \odot \Phi_{\beta}'(t) , \int^t_0 \beta y_\beta(s) \Phi_{\beta}'(s) ds \rangle dt \\
&\rightarrow \int^m_0 \langle \xi''(\Phi(t)) \odot \Phi'(t), \int^t_0 \alpha(s) \Phi'(s) ds \rangle dt, \text{ as }  \beta \rightarrow \infty.
\end{align*}

Finally,
\begin{align*}
& \left | \int^{v_{\beta}}_0 \langle [\beta (\check{\Phi}_{\beta}(m)-\check{\Phi}_{\beta}(t))]^{-1}, \Phi_{\beta}'(t) \rangle dt -  \int^{v_\beta}_0 \langle (L-\int^t_0 \alpha(s) \Phi'(s) ds)^{-1}, \Phi'(t) \rangle dt \right | \\
&\leq  \left | \int^{v_{\beta}}_0 \langle [\beta (\check{\Phi}_{\beta}(m)-\check{\Phi}_{\beta}(t))]^{-1} -  (L-\int^t_0 \alpha(s) \Phi'(s) ds)^{-1}, \Phi_{\beta}'(t) \rangle dt  \right | \\
&- \left |  \int^{v_\beta}_0 \langle (L-\int^t_0 \alpha(s) \Phi'(s) ds)^{-1}, \Phi'(t) -\Phi_{\beta}'(t) \rangle dt \right |.
\end{align*}
As $L-\int^t_0 \alpha(s) \Phi'(s) ds > \Omega$ and therefore $(L-\int^t_0 \alpha(s) \Phi'(s) ds)^{-1} > \Omega^{-1}$, we obtain that the second term in the inequality above tends to 0 as $\beta \rightarrow \infty$.\\

Now it remains to deal with the first term. Since for $t \in [0,m]$, 
\[
\lim_{\beta \rightarrow \infty} \beta (\check{\Phi}_{\beta}(m)-\check{\Phi}_{\beta}(t)) = L - \int^t_0 \alpha(s) \Phi'(s) ds, 
\] 
for $\epsilon>0$ sufficiently small, there exists $\beta_0$ large enough, such that whenever $\beta>\beta_0$, 
\[
|| (\beta (\check{\Phi}_\beta(m)-\check{\Phi}_\beta(t)) )^{-1}- (L-\int^t_0 \alpha(s) \Phi'(s) ds)^{-1} ||_\infty < \epsilon.
\]
By the definition of $\Omega$, we also get $L-\int^t_0 \alpha(s) \Phi'(s) ds > \Omega$ and then $(L-\int^t_0 \alpha(s) \Phi'(s) ds)^{-1} > \Omega^{-1}$, for $t \in [0,m)$. Moreover, by the setting of $\Phi_{\beta}$ and $\Phi$, they are uniformly bounded, $i.e.$ there exists $M>0$ such that for all $t \in [0,m]$, $|| \Phi(t) ||_\infty < M$.

Based on the relations above, we obtain
\begin{align*}
M_\beta &:= \left | \int^{v_{\beta}}_0 \langle [\beta (\check{\Phi}_{\beta}(m)-\check{\Phi}_{\beta}(t))]^{-1} -  (L-\int^t_0 \alpha(s) \Phi'(s) ds)^{-1}, \Phi_{\beta}'(t) \rangle dt  \right | \\
&\leq  \left | \int^{m}_0 \langle \mathbbm{1}_{[0,v_\beta]}(t) [\beta (\check{\Phi}_{\beta}(m)-\check{\Phi}_{\beta}(t))]^{-1} (L-\beta \check{\Phi}_{\beta}(m))  (L-\int^t_0 \alpha(s) \Phi'(s) ds)^{-1}, \Phi_{\beta}'(t) \rangle dt  \right | \\
&+ \bigg | \int^{m}_0 \langle \mathbbm{1}_{[0,v_\beta]}(t) [\beta (\check{\Phi}_{\beta}(m)-\check{\Phi}_{\beta}(t))]^{-1} \\
&\quad \times(\int^t_0 \alpha(s) \Phi'(s) ds-\beta \check{\Phi}_{\beta}(t))])  (L-\int^t_0 \alpha(s) \Phi'(s) ds)^{-1}, \Phi_{\beta}'(t) \rangle dt  \bigg |
\end{align*}
Then by dominated convergence theorem, we get that $M_\beta \rightarrow 0$ as $\beta \rightarrow \infty$.
Theorem \ref{thm:thmx0} immediately follows. 
\end{proof}

\section{Properties of the Crisanti-Sommers functional at finite temperature}\label{Sec4}

In this section, we provide  some initial properties of the Crisanti-Sommers functional at finite temperature and their minimizers. Without confusion of notations, we absorb the $\beta$ into the model $\xi$ in this section. Our starting point is a critical point equation for points in the support of the Parisi measure.

\begin{theorem}\label{thm:thm4}
Let $(x,\Phi)$ be a minimizer of \eqref{CSF}. For $t \in \text{supp } \mu_\beta$, the following equation holds:
\begin{equation*}\label{eq:charac}
\xi'(\Phi(t))+\vec{h}\vec{h}^T=\int^t_0\hat{\Phi}(s)^{-1}\Phi'(s)\hat{\Phi}(s)^{-1}ds.
\end{equation*}

\end{theorem}

\begin{proof}[Proof of Proposition \ref{thm:thm4}]

Let  $a(t)$ be any non-zero continuous function on $[0,m]$ satisfying $0 \leq t+a(t) \leq m$ for all $t \in [0,m]$ and $|a(t)-a(t')| \leq |t-t'|$ for all $t,t' \in [0,m]$. Assume that $a(0)=0$ and $a(t)=0$, for all $t \in [\hat{T},m]$. 

Set $\Theta(\epsilon,t)=(1-\epsilon)\Phi(t+a(t))+\epsilon \Psi(t+a(t))$, where $\Psi$ is an arbitrary matrix path in $\mathscr{M}$.
Since $(x,\Phi)$ is a minimizer of the Crisanti-Sommers functional
\begin{eqnarray}\label{eq:dap}
0&\leq& \partial_\epsilon \mathscr{C}(x,\Theta(\epsilon,t))\bigg|_{\epsilon=0}  \\&=& \int ^{\hat{T}}_0 \left \langle  \int ^t_0 \hat{\Phi}{(s)}^{-1}\Phi'(s)\hat{\Phi}(s)^{-1} ds-\vec{h}\vec{h}^T- \xi'(\Phi(t)),\Psi(t)-\Phi(t)+a(t) \Phi'(t) \right \rangle \mu_P(dt). \nonumber
\end{eqnarray}

Assume that the equality above holds (we provide the details of the computation of the derivative below). 
We now claim that \eqref{eq:dap} implies 
\begin{equation}\label{eq:(7)} 
\left \langle  \int ^t_0 \hat{\Phi}{(s)}^{-1}\Phi'(s)\hat{\Phi}(s)^{-1} ds-\vec{h}\vec{h}^T- \xi'(\Phi(t)),\Psi(t)-\Phi(t) \right \rangle
\end{equation}
must vanish  for all $t \neq 0$ and $t \in$ supp $\mu_P$.
Indeed, if \eqref{eq:(7)} doesn't vanish, we can always modify $\Psi(t)-\Phi(t)$ and $a(t)$ so that \eqref{eq:dap} is negative. 

 Now, we turn to the proof of \eqref{eq:dap}. By setting $\Theta(\epsilon,t):= (1-\epsilon) \Phi(t+\epsilon a(t))+\epsilon \Psi(t+\epsilon a(t))$, we start by calculating the differential of $\mathscr{C}$ at $\epsilon=0$:
\begin{align} \label{eq:123}
  2\frac{d}{d\epsilon} \mathscr{C} (x, \Theta)\bigg|_{\epsilon=0} &= \frac{d}{d\epsilon}\bigg[\int^n_0x(t) \langle \xi'(\Theta(t))+\vec{h}\vec{h}^T, \Theta'(t) \rangle dt  \nonumber \\ 
&+ \log|\Theta(n)-\Theta(\hat{T})|+ \int^{\hat{T}}_0 \langle \hat{\Theta}(t)^{-1},\Theta'(t) \rangle dt]\bigg ] \bigg |_{\epsilon=0}  \nonumber \\
&:=I+II+III.
 \end{align}
Next, we notice that 
\begin{equation}\label{eq:firstderH}
\frac{d}{d\epsilon} \Theta(t)\bigg|_{\epsilon=0}=\Psi(t)-\Phi(t)+a(t)\Phi'(t) 
\end{equation}
and 
\begin{equation}\label{eq:secondderH}
\frac{d}{d\epsilon} \Theta'(t)\bigg|_{\epsilon=0}=(\Psi(t)-\Phi(t)+a(t)\Phi'(t))'.
\end{equation}

Thus, combining \eqref{eq:123}, \eqref{eq:firstderH}, and \eqref{eq:secondderH}, we obtain
\begin{align}\label{eq:one}
I&=\int^{\hat{T}}_0x(t) \langle \xi''(\Phi(t) \odot \Phi'(t), \Psi(t)-\Phi(t)+a(t)\Phi'(t) \rangle dt \nonumber \\
     &+ \int^{\hat{T}}_0 x(t) \langle \xi'(\Phi(t))+\vec{h}\vec{h}^T,(\Psi(t)-\Phi(t)+a(t)\Phi'(t))' \rangle dt,
\end{align}
and
 \begin{align}\label{eq:one22}
 II&=\left \langle (\Phi(m)-\Phi(\hat{T}))^{-1}, \Psi(m)-\Phi(m)+a(m)\Phi'(m) \right \rangle \nonumber \\
    &- \left \langle (\Phi(m)-\Phi(\hat{T}))^{-1}, \Psi(\hat{T})-\Phi(\hat{T})+a(\hat{T})\Phi'(\hat{T}) \right \rangle=0,
    \end{align}
and, by integration by parts,
\begin{align}\label{eq:one223}
 III&=- \int^{\hat{T}}_0 \int^m_tx(s) \left \langle \hat{\Phi}(t)^{-1}\Phi'(t) \hat{\Phi}(t)^{-1}, ( \Psi(s)-\Phi(s)+a(s)\Phi'(s) )'\right \rangle ds dt \nonumber \\
    &+\int^{\hat{T}}_0 \left \langle \hat{\Phi}(t)^{-1},( \Psi(t)-\Phi(t)+a(t)\Phi'(t) )'  \right\rangle dt  \nonumber \\
  &=-\int^{\hat{T}}_0 x(t) \left \langle (\Psi(t)-\Phi(t)+a(t)\Phi'(t))',\int^t_0 \hat{\Phi}(t)^{-1} \Phi'(t) \hat{\Phi}(t)^{-1}ds\right \rangle dt \nonumber \\
   &-\int^{\hat{T}}_0 x(t) \left \langle \Psi(t)-\Phi(t)+a(t)\Phi'(t),\hat{\Phi}(t)^{-1}\Phi'(t)\hat{\Phi}(t)^{-1}\right \rangle dt.
\end{align}
Here, we use the fact that $\Psi(t)=\Phi(t)$ for $t \in [\hat{T},m]$ and $a(0)=0$ and $a(t)=0$, for all $t \in [\hat{T},m]$.

By adding \eqref{eq:one}, \eqref{eq:one22}, and \eqref{eq:one223}, and using \eqref{eq:123} we get
\begin{align*}
\frac{d}{d\epsilon} &\mathscr{C} (x, \Theta) \bigg|_{\epsilon=0}=\frac{1}{2}\bigg[\int^{\hat{T}}_0 x(t)\left\langle \Phi'(t) \odot \xi''(\Phi(t))-\hat{\Phi}(t)^{-1}\Phi'(t)\hat{\Phi}(t)^{-1},\Psi(t)-\Phi(t)+a(t)\Phi(t) \right\rangle dt\\
&+\int^{\hat{T}}_0 x(t)\left \langle \vec{h}\vec{h}^T+\xi'(\Phi(t))-\int^t_0\hat{\Phi}(s)^{-1}\Phi'(s)\hat{\Phi}(s)^{-1}ds,(\Psi(t)-\Phi(t)+a(t)\Phi'(t))'\right \rangle dt\bigg]\\
&=\frac{1}{2}\bigg[\int^{\hat{T}}_0 x(t)(\left\langle \vec{h}\vec{h}^T+\xi'(\Phi(t))-\int^t_0\hat{\Phi}(s)^{-1}\Phi'(s)\hat{\Phi}(s)^{-1}ds,\Psi(t)-\Phi(t)+a(t)\Phi(t) \right\rangle)' dt\bigg] \\
&=\frac{1}{2}\bigg[\int^{\hat{T}}_0 \left\langle \vec{h}\vec{h}^T+\xi'(\Phi(t))-\int^t_0\hat{\Phi}(s)^{-1}\Phi'(s)\hat{\Phi}(s)^{-1}ds, \Psi(t)-\Phi(t)+a(t)\Phi'(t)  \right \rangle \mu_P(dt)\bigg].\\
\end{align*}
The critical point condition $\frac{d}{d\epsilon} \mathscr{C} (x, \Theta) |_{\epsilon=0} \geq 0$ combined with the above equation proves \eqref{eq:dap}.

\end{proof}

\begin{proposition}\label{thm:thm2}
If $(a,b) \subseteq supp \; \mu_P$ with $0 \leq a < b < m$, then
\begin{equation*}
\mu_P([0,u])=\frac{\langle \xi'''(\Phi(u)) , \Phi'(u)^{\circ 3} \rangle }{2\text{Trace}\left((\Phi'(u)^{\frac{1}{2}}(\xi''(\Phi(u)) \odot \Phi'(u))\Phi'(u)^{\frac{1}{2}}))^{\frac{3}{2}}\right)}
\end{equation*}
for all $u \in (a,b)$.
\end{proposition}

\begin{proof}[Proof of Proposition \ref{thm:thm2}] 
First note that as $(a,b) \in supp \; \mu$, then by Theorem \ref{thm:thm4},  for all $u \in (a,b)$,
\begin{equation*}
F(u):=\xi'(\Phi(u))+\vec{h}\vec{h}^T-\int^u_0\hat{\Phi}(s)^{-1}\Phi'(s)\hat{\Phi}(s)^{-1}ds=0,
\end{equation*}
and as $\Phi$ is differentiable for all $u \in (a,b)$, we can differentiate the equation above and obtain that, for all $u \in (a,b)$
\begin{equation}\label{eq:isthistrue}
 G(u):=\xi''(\Phi(u)) \odot \Phi'(u)- \hat{\Phi}{(u)}^{-1}\Phi'(u)\hat{\Phi}(u)^{-1}=0.
\end{equation}

From \eqref{eq:isthistrue} and the fact that $\Phi'$ is positive semi-definite, we obtain
 \[
 \Phi'(u)^{\frac{1}{2}}( \xi''(\Phi(u)) \odot \Phi'(u)- \hat{\Phi}{(u)}^{-1}\Phi'(u)\hat{\Phi}(u)^{-1})\Phi'(u)^{\frac{1}{2}}=0,
 \]
 which implies that 
\begin{equation}\label{eq:squareroot} 
\Phi'(u)^{\frac{1}{2}}[\xi''(\Phi(u)) \odot \Phi'(u)]\Phi'(u)^{\frac{1}{2}}=(\Phi'(u)^{\frac{1}{2}}\hat{\Phi}(u)^{-1}\Phi'(u)^{\frac{1}{2}})^2.
\end{equation}
As both sides of the equality in \eqref{eq:squareroot} are positive semi-definite matrices, their positive semi-definite square root coincide. We thus get the relation
\begin{equation*}
\Phi'(u)^{-\frac{1}{2}}(\Phi'(u)^{\frac{1}{2}}(\xi''(\Phi(u)) \odot \Phi'(u))\Phi'(u)^{\frac{1}{2}})^{\frac{1}{2}}\Phi'(u)^{-\frac{1}{2}}=\hat{\Phi}(u)^{-1}.
\end{equation*}

As we know that $(a,b) \in \text{supp }\mu_P$ and $\Phi$ is twice differentiable in $(a,b)$, therefore $G''(u)=0$. Consider the Frobenius inner product on $G''(u)$ and $\Phi'(u)$,
\begin{align*}
0&= \langle G''(u), \Phi'(u) \rangle \\
&= \langle \xi'''(\Phi(u))\odot \Phi'(u)^{\circ2} , \Phi'(u) \rangle -2\mu_P([0,u]) \langle \hat{\Phi}(u)^{-1}\Phi'(u_0)\hat{\Phi}(u)^{-1}\Phi'(u)\hat{\Phi}(u)^{-1} , \Phi'(u) \rangle\\
   &+ \langle \xi''(\Phi(u))\odot \Phi''(u)-\hat{\Phi}(u)^{-1}\Phi''(u)\hat{\Phi}(u)^{-1} , \Phi'(u) \rangle\\
&=\langle \xi'''(\Phi(u)), \Phi'(u)^{\circ3}  \rangle -2\mu_P([0,u]) \langle \hat{\Phi}(u)^{-1}\Phi'(u)\hat{\Phi}(u)^{-1}\Phi'(u)\hat{\Phi}(u)^{-1} , \Phi'(u) \rangle\\
   &+ \langle \xi''(\Phi(u))\odot \Phi'(u)-\hat{\Phi}(u)^{-1}\Phi'(u)\hat{\Phi}(u)^{-1} , \Phi''(u) \rangle\\
&=\langle \xi'''(\Phi(u)), \Phi'(u)^{\circ3}  \rangle -2\mu_P([0,u]) \langle \hat{\Phi}(u)^{-1}\Phi'(u)\hat{\Phi}(u)^{-1}\Phi'(u)\hat{\Phi}(u)^{-1} , \Phi'(u) \rangle\\
&=\langle \xi'''(\Phi(u)) , \Phi'(u)^{\circ3}  \rangle  -2\mu_P([0,u]) \text{Trace}((\Phi'(u)^{\frac{1}{2}}(\xi''(\Phi(u)) \odot \Phi'(u))\Phi'(u)^{\frac{1}{2}}))^{\frac{3}{2}}).\\
\end{align*}
Therefore we get the desired conclusion.
\end{proof}

The next result shows that $0$ can only be isolated point of the Parisi measure if the model does not have an SK component.

\begin{proposition}
If $2 \langle \vec{\beta}_2 \otimes \vec{\beta}_2 , \Phi'(0)^{\circ 2} \rangle \neq \langle \hat{\Phi}(0)^{-1}\Phi'(0),\hat{\Phi}(0)^{-1} \Phi'(0) \rangle$ and $0 \in$ supp $\mu_P$, then there exists $\hat{q} > 0$ such that $\mu_P ([0, \hat{q}]) =\mu_P (\{0\}).$
\end{proposition}

\begin{proof}
Suppose $0 \in$ supp $\mu_P$ and the existence of a sequence of $u_n \downarrow 0$ such that $u_n \in$ supp $\mu_P.$ 

By considering the function $f:[0,m]\to \mathbb{R}$ and $g:[0,m]\to \mathbb{R}$ given respectively by
\begin{equation*}\label{def:f}
f(u)=\langle \xi'(\Phi(u))+\vec{h}\vec{h}^T-\int^u_0\hat{\Phi}(s)^{-1}\Phi'(s)\hat{\Phi}(s)^{-1}ds, \Phi'(u) \rangle
\end{equation*}
and
\begin{align*}\label{def:g}
g(u)&= \langle \xi''(\Phi(u)) \odot \Phi'(u) - \hat{\Phi}(u)^{-1}\Phi'(u)\hat{\Phi}(u)^{-1}, \Phi'(u) \rangle\\
&+\langle \xi'(\Phi(u))+\vec{h}\vec{h}^T-\int^u_0\hat{\Phi}(s)^{-1}\Phi'(s)\hat{\Phi}(s)^{-1}ds, \Phi''(u) \rangle.
\end{align*}
Then $f(u_n)=0$ for all $n \geq 1.$
By mean value theorem, there exists a sequence $u'_n \downarrow 0$ such that $f'(u_n)=0$. Notice that $f'(u_n)=g(u_n).$ By the continuity of $g$ at 0, we obtain $g(0)=0$, which implies that $2 \langle \vec{\beta}_2 \otimes \vec{\beta}_2 , \Phi'(0)^{\circ 2} \rangle=\langle \hat{\Phi}(0)^{-1}\Phi'(0),\hat{\Phi}(0)^{-1} \Phi'(0) \rangle$.
\end{proof}

The next proposition is the high-dimensional analogue of Theorem from \cite{1}.

\begin{proposition}
Suppose that there exist an increasing sequence $(u^-_l)_{l \geq 1}$ and a decreasing sequence $(u^+_l)_{l \geq 1}$ of supp $\mu_P$ such that $\lim_{l \rightarrow \infty} u_l^- =u_0 = \lim_{l \rightarrow \infty} u^+_l$. Then $\mu_P$ is continuous at $u_0$.
\end{proposition}

\begin{proof}
Since both $(u^+_l)_{l \geq 1}$ and $(u^-_l)_{l \geq 1}$ converges to $u_0$ and lies in supp $\mu_P$, by mean value theorem and continuity of $g$, we obtain
\begin{align*}
&0= \lim_{h \rightarrow 0^+} \frac{g(u_0+h)-g(u_0)}{h}\\
&=\bigg\langle \xi'''(\Phi(u_0)) \odot \Phi'(u_0)^{\circ2} +\xi''(\Phi(u)) \odot \Phi''(u_0) \\
&\quad  \quad - 2 \mu_P([0,u_0])  \langle \hat{\Phi}(u_0)^{-1}\Phi'(u_0) \hat{\Phi}(u_0)^{-1}\Phi'(u_0)\hat{\Phi}(u_0)^{-1}, \Phi'(u_0) \bigg\rangle\\
&+\langle \xi'(\Phi(u_0))+\vec{h}\vec{h}^T-\int^{u_0}_0\hat{\Phi}(s)^{-1}\Phi'(s)\hat{\Phi}(s)^{-1}ds, \Phi'''(u_0) \rangle \\
&+2\langle \xi''(\Phi(u_0)) \odot \Phi'(u_0) - \hat{\Phi}(u_0)^{-1}\Phi'(u_0)\hat{\Phi}(u_0)^{-1}, \Phi''(u_0) \rangle.
\end{align*}

\begin{align*}
&0= \lim_{h \rightarrow 0^-} \frac{g(u_0+h)-g(u_0)}{h}\\
&=\bigg \langle \xi'''(\Phi(u_0)) \odot \Phi'(u_0)^{\circ2} +\xi''(\Phi(u)) \odot \Phi''(u_0)
\\
&\quad \quad - 2 \mu_P([0,u_0))  \langle \hat{\Phi}(u_0)^{-1}\Phi'(u_0) \hat{\Phi}(u_0)^{-1}\Phi'(u_0)\hat{\Phi}(u_0)^{-1}, \Phi'(u_0) \bigg \rangle\\
&+\langle \xi'(\Phi(u_0))+\vec{h}\vec{h}^T-\int^{u_0}_0\hat{\Phi}(s)^{-1}\Phi'(s)\hat{\Phi}(s)^{-1}ds, \Phi'''(u_0) \rangle \\
&+2\langle \xi''(\Phi(u_0)) \odot \Phi'(u_0) - \hat{\Phi}(u_0)^{-1}\Phi'(u_0)\hat{\Phi}(u_0)^{-1}, \Phi''(u_0) \rangle.
\end{align*}
By comparing the two equations, we obtain $ \mu_P([0,u_0])= \mu_P([0,u_0))$, which implies that $\mu_P$ is continuous at $u_0$.

\end{proof}

\begin{proposition}
For any $u_0 \in$ supp $\mu_P$, if $\langle \xi''(Q), \Phi'(u_0)^{\circ 2} \rangle < \langle Q^{-1} \Phi'(u_0), Q^{-1} \Phi'(u_0) \rangle$, then the Parisi measure $\mu_P$ has a jump discontinuity at $u_0$.
\end{proposition}
\begin{proof}
If $u_0$ is an isolated point of supp $\mu_P$, it must be a jump discontinuity of $\mu_P$. Now assume that $u_0$ is not isolated and $\mu_P$ is continuous at the point $u_0$. Then by Theorem \ref{thm:thm4} and the mean value theorem we obtain
\begin{equation*}
\langle \xi''(\Phi(u_0)),\Phi'(u_0)^{\circ 2} \rangle = \langle \hat{\Phi}(u_0)^{-1} \Phi'(u_0),\hat{\Phi}(u_0)^{-1} \Phi'(u_0) \rangle .
\end{equation*}

As $\hat{\Phi}(u_0)=\int^m_{u_0} x(t) \Phi'(t) dt$, which implies that $\hat{\Phi}(u_0)^{-1} \geq (Q-\Phi(u_0))^{-1} \geq Q^{-1}$, we have 
\begin{eqnarray*}
\langle \xi''(\Phi(Q)),\Phi'(u_0)^{\circ 2} \rangle  &\geq& \langle \xi''(\Phi(u_0)),\Phi'(u_0)^{\circ 2} \rangle \\
&\geq&  \langle (Q-\Phi(u_0))^{-1} \Phi'(u_0),(Q-\Phi(u_0))^{-1} \Phi'(u_0) \rangle \\
&\geq& \langle Q^{-1} \Phi'(u_0),Q^{-1} \Phi'(u_0) \rangle
\end{eqnarray*}
which leads to a contradiction.
\end{proof}

We say two points $x, y$  are consecutive isolated points of the support of a measure $\mu$ if $x,y$ are isolated points in $\text{supp } \mu$ with $x<y$ and $\mu((x,y))=0$.

\begin{proposition}
Let $(x,\Phi)$ be a minimizer of \eqref{eq:CSFormula}. If $q \mapsto \langle \xi''(\Phi(t)), \Phi'(t)^{\circ 2} \rangle ^{-\frac{1}{2}}$ is convex in an interval I, then supp $\mu_P$ contains at most 2 consecutive isolated points. In particular, if $\xi(A)=(\vec{\beta}_p \otimes \vec{\beta}_p) \odot A^{\circ p}$, then supp $\mu_P$ contains at most 2 consecutive isolated points.
\end{proposition}

\begin{proof}
As $\mu_P$ is non-decreasing, it contains at most countably many atoms. Now assume supp $\mu_P$ contains countably many isolated points, and we connect the points in supp $\mu_P$ by linear interpolation. Then for two consecutive points $s_1<s_2$ in supp $\mu_P$, define $a_0 : = \mu_P([0,s_1])$.

Furthermore, we obtain that for $u \in (s_1,s_2)$, $\Phi''(u)=0$ and therefore
\begin{equation*}
g(u)= \langle \xi''(\Phi(u)) \odot \Phi'(u) - \hat{\Phi}(u)^{-1}\Phi'(u)\hat{\Phi}(u)^{-1}, \Phi'(u) \rangle.
\end{equation*}
We also have
\begin{equation*}
f(s_1)=f(s_2)=0 \text{ and } \int^{s_2}_{s_1} f(q) dq =0
\end{equation*}
where\begin{equation*}
f(u)=\langle \xi'(\Phi(u))+\vec{h}\vec{h}^T-\int^u_0\hat{\Phi}(s)^{-1}\Phi'(s)\hat{\Phi}(s)^{-1}ds, \Phi'(u). \rangle
\end{equation*}
Therefore $g(u)=0$ has 2 solutions between $s_1$ and $s_2.$

The relation $g(u)=0$ is equivalent to say that
\begin{equation*}
\langle \xi''(\Phi(u)) , \Phi'(u)^{\circ 2} \rangle = \langle \hat{\Phi}(u)^{-1}\Phi'(u)\hat{\Phi}(u)^{-1}, \Phi'(u) \rangle.
\end{equation*}

We then define $y(u)=\langle \xi''(\Phi(u)) , \Phi'(u)^{\circ 2} \rangle^{-\frac{1}{2}} $ and $z(u)= \langle \hat{\Phi}(u)^{-1}\Phi'(u)\hat{\Phi}(u)^{-1}, \Phi'(u) \rangle^{-\frac{1}{2}}$. Thus
\begin{eqnarray*}
z'(u)=- a_0 \langle \hat{\Phi}(t)^{-1}\Phi'(t) , \hat{\Phi}(t)^{-1}\Phi'(t) \rangle^{-\frac{3}{2}} \cdot \langle \hat{\Phi}(t)^{-1}\Phi'(t)\hat{\Phi}(t)^{-1}\Phi'(t),\hat{\Phi}(t)^{-1}\Phi'(t) \rangle
\end{eqnarray*}
and
\begin{eqnarray*}
z''(u)&=& 3 a^2_0 \langle \hat{\Phi}(t)^{-1}\Phi'(t) , \hat{\Phi}(t)^{-1}\Phi'(t) \rangle^{-\frac{5}{2}} \cdot \langle \hat{\Phi}(t)^{-1}\Phi'(t)\hat{\Phi}(t)^{-1}\Phi'(t),\hat{\Phi}(t)^{-1}\Phi'(t) \rangle^2\\
&&-3a^2_0 \langle \hat{\Phi}(t)^{-1}\Phi'(t) , \hat{\Phi}(t)^{-1}\Phi'(t) \rangle^{-\frac{3}{2}} \cdot \langle \hat{\Phi}(t)^{-1}\Phi'(t)\hat{\Phi}(t)^{-1}\Phi'(t),\hat{\Phi}(t)^{-1}\Phi'(t) \hat{\Phi}(t)^{-1}\Phi'(t) \rangle \\
&=&3a_0^2 \langle \hat{\Phi}(t)^{-1}\Phi'(t) , \hat{\Phi}(t)^{-1}\Phi'(t) \rangle^{-\frac{5}{2}} \cdot (\langle \hat{\Phi}(t)^{-1}\Phi'(t)\hat{\Phi}(t)^{-1}\Phi'(t),\hat{\Phi}(t)^{-1}\Phi'(t) \rangle^2\\
&&-  \langle \hat{\Phi}(t)^{-1}\Phi'(t)\hat{\Phi}(t)^{-1}\Phi'(t),\hat{\Phi}(t)^{-1}\Phi'(t) \hat{\Phi}(t)^{-1}\Phi'(t) \rangle \cdot \langle \hat{\Phi}(t)^{-1}\Phi'(t) , \hat{\Phi}(t)^{-1}\Phi'(t) \rangle) \leq 0,
\end{eqnarray*}
where the last step uses the Cauchy inequality.
Hence $z(u)$ is a concave function of $u$, this equation can have at most 2 roots in any interval where $\langle \xi''(\Phi(t)), \Phi'(t)^{\circ 2} \rangle ^{-\frac{1}{2}}$ is convex.

Now to see that $\langle \xi''(\Phi(t)), \Phi'(t)^{\circ 2} \rangle ^{-\frac{1}{2}}$ is convex for $\xi(A)=(\beta_p \otimes \beta_p) \odot A^{\circ p}$, we need to modify $\Phi$ without the change of $\mathscr{C}(x,\Phi).$ 

For any $\epsilon>0$, set 

 \begin{equation*} \Phi_\epsilon(t)=
 \left\{
\begin{array}{lcl}
\frac{t}{s_2} \cdot \Phi(s_2) \text{ for } t \in (s_1+\epsilon , s_2)   \\
\text{some smooth curve that connects } \Phi(s_1) \text{ and } \Phi(s_1+\epsilon)  \text{ for } t \in (s_1 , s_1+\epsilon). \\
\end{array} \right. \end{equation*} 
Here we set $\Phi_\epsilon$ mild enough so that $\hat{\Phi}_\epsilon$ is positive definite on $[s_1,s_1+\epsilon].$ Note that $(x,\Phi_\epsilon)$ is still a minimizer of $\mathscr{C}.$

Therefore for $u \in [s_1+\epsilon,s_2],$
\begin{eqnarray*}
y'(u)=-\frac{1}{2} \langle \xi''(\Phi(u)),\Phi'(u)^{\circ 2} \rangle^{-\frac{3}{2}} \cdot \langle \xi'''(\Phi(u)),\Phi'(u)^{\circ 3} \rangle
\end{eqnarray*}
and
\begin{align*}
y''(u)&=-\frac{1}{2} \langle \xi''(\Phi(u)),\Phi'(u)^{\circ 2} \rangle^{-\frac{3}{2}} \cdot \langle \xi''''(\Phi(u)),\Phi'(u)^{\circ 4} \rangle 
\\&\quad +\frac{3}{4}  \langle \xi''(\Phi(u)),\Phi'(u)^{\circ 2} \rangle^{-\frac{5}{2}} \cdot \langle \xi'''(\Phi(u)),\Phi'(u)^{\circ 3} \rangle^2 \\
&=\frac{1}{4} \langle \xi''(\Phi(u)),\Phi'(u)^{\circ 2} \rangle^{-\frac{5}{2}} \cdot [3\langle \xi'''(\Phi(u)),\Phi'(u)^{\circ 3} \rangle^2
 \\& \quad- 2\langle \xi''''(\Phi(u)),\Phi'(u)^{\circ 4} \rangle \cdot \langle \xi''(\Phi(u)),\Phi'(u)^{\circ 2} \rangle]\\
&=\frac{1}{4} \langle \xi''(\Phi(u)),\Phi'(u)^{\circ 2} \rangle^{-\frac{5}{2}} \cdot [3\langle \xi'''(\Phi(u)),\Phi'(u)^{\circ 3} \rangle^2
\\ & \quad- 2\langle \xi''''(\Phi(u)),\Phi'(u)^{\circ 4} \rangle \cdot \langle \xi''(\Phi(u)),\Phi'(u)^{\circ 2} \rangle]\\
&= \frac{1}{4s^6_2} p^2(p-1)^2(p-2)  \langle \xi''(\Phi(u)),\Phi'(u)^{\circ 2} \rangle^{-\frac{5}{2}}\\
&\quad \cdot [ (3p-6) t^{2p-6} \langle \beta_p \otimes \beta_p, \Phi(s_2)^{\circ p} \rangle ^2 -(2p-6) t^{2p-6} \langle \beta_p \otimes \beta_p, \Phi(s_2)^{\circ p} \rangle^2]\\
&= \frac{1}{4s^6_2} p^3(p-1)^2(p-2) t^{3p-8}  \langle \beta_p \otimes \beta_p ,\Phi(s_2)^{\circ p} \rangle^{-\frac{5}{2}}\cdot  \langle \beta_p \otimes \beta_p, \Phi(s_2)^{\circ p} \rangle ^2 \geq 0,
\end{align*}
which means that $y(u)$ is convex on $[s_1+\epsilon,s_2].$ We then let $\epsilon$ tend to 0 and then get the desired conclusion.

\end{proof}
\subsubsection{A full RSB example}
Consider the constraint 
\begin{equation*}
Q=\left(                
  \begin{array}{ccc}   
    1& 0.1\\ 
    0.1 &1\\  
  \end{array}
\right).
\end{equation*}
For $\beta > \sqrt{\frac{2}{\text{tr}(Q^2)}}$, define $\vec{\beta}:=(\beta,\beta)$. For any matrix $A$, define $chA:(ch(a_{i,j}))_{1\leq i,j \leq 2}$ and $shA:(sh(a_{i,j}))_{1\leq i,j \leq 2}$. Moreover, define $E=\vec{e} \otimes \vec{e}$, where $\vec{e}=(1,1)$.
We now set the model with $\vec{h}=0$ and
\begin{eqnarray*}
\xi(A)=(\vec{\beta}\otimes \vec{\beta}) \odot (chA-E)
\end{eqnarray*}
Define a matrix path as $\Phi(q)=\frac{q}{2}Q$ and $\phi(q)=\langle \xi''(\Phi(q)),\Phi'(q)^{\circ 2} \rangle ^{-\frac{1}{2}}$.

\begin{eqnarray*}
\phi''(q)=\frac{1}{4} \langle \xi''(\Phi(q)),\Phi'(q)^{\circ2} \rangle^{-\frac{5}{2}}[3\langle \xi'''(\Phi(q)),\Phi'(q)^{\circ3} \rangle^2-2\langle \xi''(\Phi(q)),\Phi'(q)^{\circ 2} \rangle \langle \xi''''(\Phi(q)), \Phi'(q)^{\circ4} \rangle].
\end{eqnarray*}
Since the copies are at the same temperature, the conclusion that $\phi''(q)<0$ can be deduced by the fact that $th^2(1.1)<\frac{2}{3}$ and then $3sh^2q<2ch^2q,$ for $q<1.1$. Hence $\phi(q)$ is concave on $[0,2].$ 

Since $\phi(0)<\sqrt{2}$ and $\phi(2)>0$, there exists a unique $q_0$ such that $\phi(q_0)=\frac{1}{\sqrt{2}}(2-q_0).$ Define a distribution function as follows:
 \begin{equation*} x(q)= \left\{
\begin{array}{lcl}
-\sqrt{2}\phi'(q)  \qquad 0 \leq q \leq q_0   \\
1   \qquad q \geq q_0. \\
\end{array} \right. \end{equation*} 

We claim that $(x,\Phi)$ is a minimizer of $\mathscr{C}$ for the model $\xi$. 
\begin{proof}
Based on the definition of $x$ and $\Phi,$ for $0 \leq q \leq q_0$ we obtain
\begin{eqnarray*}
\hat{\Phi}(q)&=&Q-\Phi(q_0)-\sqrt{2}\phi(q_0)\frac{1}{2}Q+\sqrt{2}\phi(q)\frac{1}{2}Q \\
&=&Q-\frac{q_0}{2}Q-\frac{1}{2}(2-q_0)Q+\frac{\sqrt{2}}{2}\phi(q)Q \\
&=&\sqrt{2}\phi(q)\Phi'(q)
\end{eqnarray*}
therefore
\begin{equation*}
\langle \hat{\Phi}(q)^{-1} \Phi'(q) , \hat{\Phi}(q)^{-1} \Phi'(q) \rangle=\phi(q)^{-2}=\langle \xi''(\Phi(q)),\Phi'(q)^{\circ 2} \rangle.
\end{equation*}

Recall the function \begin{equation*}
f(u)=\langle \xi'(\Phi(u))+\vec{h}\vec{h}^T-\int^u_0\hat{\Phi}(s)^{-1}\Phi'(s)\hat{\Phi}(s)^{-1}ds, \Phi'(u) \rangle.
\end{equation*}
Since $f(0)=0$, we then obtain that $f(q)=0$ for $0 \leq q \leq q_0.$
Moreover, since 
\[
\phi(q)>\frac{1}{\sqrt{2}}(2-q)=\langle (Q-\Phi(q))^{-1} \Phi'(q) , (Q-\Phi(q))^{-1} \Phi'(q) \rangle,\]
we obtain 
\begin{equation*}
\langle \xi''(\Phi(q)),\Phi'(q)^{\circ 2} \rangle=\phi(q)^{-\frac{1}{2}}<\langle (Q-\Phi(q))^{-1} \Phi'(q) , (Q-\Phi(q))^{-1} \Phi'(q) \rangle.
\end{equation*}
Therefore if we define $h(s)=\int^s_0 f(q) dq$, we get that $h(s)=0$ for $s<q_0$ and $f(s)<0$ for $s>q_0.$

Thus $(x,\Phi)$ minimizes the Crisanti-Sommers functional $\mathscr{C}$ in this case. 
\end{proof}

\section{Properties of the Crisanti-Sommers functional at zero temperature}\label{Sec5}

In this section we provide some useful characterizations of minimizers of \eqref{eq:GSEW}.
\begin{theorem}
Let $(L, \alpha, \Phi) \in \mathscr{K}(Q)$. Define 
\begin{eqnarray*}
g(t):=\int^m_t \bar{g}(s) ds
\end{eqnarray*}
and
\begin{eqnarray*}
\bar{g}(t)=\langle \Phi'(t), \bar{G}(t) \rangle
\end{eqnarray*}

where
\begin{eqnarray*}
\bar{G}(t)=\xi'(\Phi(t))-\int^t_0  (L-\int^s_0 \alpha(q) \Phi'(q) dq )^{-1} \Phi'(s) (L-\int^s_0 \alpha(q) \Phi'(q) dq )^{-1}ds.
\end{eqnarray*}

Then $(L,\alpha,\Phi)$ is the minimizer of $\mathscr{C}$ if and only if the following equation holds,
\begin{equation*}
\xi'(Q)+\vec{h}\vec{h}^T=\int^m_0 (L-\int^t_0 \alpha(s) \Phi'(s)ds)^{-1} \Phi'(t) (L- \int^t_0 \alpha(s) \Phi'(s) ds )^{-1} dt
\end{equation*}
and the function g satisfies $\min_{u \in [0,m]} g(u) \geq 0$ and $\gamma_0(S)=\gamma_0([0,m))$, where $S:= \{ u \in [0,m) | g(u) =0 \}$.

\end{theorem}

\begin{proof}
The proof of this theorem is standard. Assume $(L_0, \alpha_0,\Phi_0)$ is minimizer of $\mathscr{C} (L,\alpha,\Phi)$, $\gamma_0$ is the corresponding measure induced by $\alpha(s)$ and consider any $(L,\alpha,\Phi) \in \mathscr{K}$. For $0 \leq \theta \leq 1$, $(L_\theta, \alpha_\theta, \Phi_\theta)$ also lies in $\mathscr{K}$, where $L_\theta =(1-\theta)L_0+\theta L, \alpha_\theta=(1-\theta)\alpha_0 +\theta \alpha$ and $\Phi_\theta=(1-\theta)\Phi_0+\theta \Phi$.

As $(L_0, \alpha_0,\Phi_0)$ minimizes $\mathscr{C}$, we obtain
\begin{align*}
&\frac{\partial \mathscr{C}(L_\theta,\alpha_\theta,\Phi_\theta)}{\partial \theta} \big | _{\theta=0}=\frac{1}{2}  \bigg[ \langle \xi'(Q)+\vec{h}\vec{h}^T,L-L_0 \rangle 
\\&+ \int^m_0 \langle (L-L_0-\int^t_0 (\alpha(s)-\alpha_0(s) )\Phi_0'(s)ds)^{-1} ,\Phi_0'(t) \rangle dt \\
&- \int^m_0 \langle \xi''(\Phi_0(t)) \odot \Phi_0'(t), \int^t_0 (\alpha(s)-\alpha_0(s)) \Phi_0'(s) ds \rangle dt \\
&+ \int^m_0 \langle (L-\int^t_0 \alpha(s) \Phi_0'(s)ds )^{-1}, \Phi'(t)-\Phi_0'(t) \rangle dt\\
&- \int^m_0 \langle (L-\int^t_0 \alpha(s) \Phi_0'(s)ds )^{-1} \Phi'(t) (L-\int^t_0 \alpha_0(s) \Phi'(s)ds)^{-1} , \int^t_0 \alpha(s) (\Phi'(s)-\Phi_0'(s))ds \rangle dt\\
&-\int^m_0 \langle \xi'''(\Phi_0(t)) \odot (\Phi(t)-\Phi_0(t)) \odot \Phi_0'(t)+ \xi''(\Phi_0(t)) \odot (\Phi'(t)-\Phi_0'(t)), \int^t_0 \alpha(s) \Phi_0'(s)ds \rangle dt \bigg]\\
 &\geq 0.
\end{align*}
Extra algebra leads to: 
\begin{align*}
&\frac{\partial \mathscr{C}(L_\theta,\alpha_\theta,\Phi_\theta)}{\partial \theta} \big | _{\theta=0}\\
&=\frac{1}{2} \bigg[ \langle L-L_0, \xi'(Q)+\vec{h}\vec{h}^T-\int^m_0 (L_0-\int^t_0 \alpha_0(s) \Phi_0'(s)ds)^{-1} \Phi_0'(t) (L_0- \int^t_0 \alpha_0(s) \Phi_0'(s) ds )^{-1} dt \rangle\\
&+\int^m_0 \langle \int^t_0 (\alpha(s) -\alpha_0(s)) \Phi_0'(s)ds, (L-\int^t_0 \alpha_0(s) \Phi'(s)ds )^{-1}\Phi'_0(t)(L-\int^t_0 \alpha_0(s) \Phi_0'(s)ds)^{-1}\rangle dt\\
&-\int^m_0 \langle \int^t_0 (\alpha(s) -\alpha_0(s)) \Phi_0'(s)ds,\xi''(\Phi_0(t) \odot \Phi_0'(t) \rangle dt\\
&+\int^m_0 \langle (L_0- \int^t_0 \alpha_0(s) \Phi'_0(s)ds)^{-1} \Phi'_0(t) (L_0-\int^t_0\alpha_0 \Phi'_0(s)ds )^{-1}, \int^t_0 \alpha_0(s) (\Phi'(s)-\Phi'_0(s))ds \rangle dt \\
&- \int^m_0 \langle \xi''(\Phi_0(t)) \odot \Phi'_0(t),\int^t_0 \alpha_0(s) (\Phi'(s)-\Phi'_0(s))ds \rangle dt\\
&+\int^m_0 \langle (L_0- \int^t_0 \alpha_0(s) \Phi'_0(s)ds)^{-1} , \Phi'(t) -\Phi'_0(t) \rangle dt \\
&- \int^m_0 \langle \xi'''(\Phi_0(t)) \odot (\Phi(t)-\Phi_0(t)) \odot \Phi'(t) + \xi''(\Phi_0(t)) \odot (\Phi'(t)-\Phi'_0(t)), \int^t_0 \alpha_0(s) \Phi_0'(s)ds \rangle dt \bigg].
\end{align*}
Based on the arbitrariness of $\alpha, L$ and $\Phi$, we obtain the the relations:
\begin{eqnarray*}
\xi'(Q)+\vec{h}\vec{h}^T=\int^m_0 (L_0-\int^t_0 \alpha_0(s) \Phi_0'(s)ds)^{-1} \Phi_0'(t) (L_0- \int^t_0 \alpha_0(s) \Phi_0'(s) ds )^{-1} dt,
\end{eqnarray*}
\begin{eqnarray*}
\int^m_0 \langle \int^t_0 (\alpha(s) -\alpha_0(s)) \Phi_0'(s)ds, (L-\int^t_0 \alpha_0(s) \Phi'(s)ds )^{-1}\Phi'_0(t)(L-\int^t_0 \alpha_0(s) \Phi_0'(s)ds)^{-1} \rangle\\
-\int^m_0 \langle \int^t_0 (\alpha(s) -\alpha_0(s)) \Phi_0'(s)ds, \xi''(\Phi_0(t) \odot \Phi_0'(t) \rangle dt \geq 0,
\end{eqnarray*}
and
\begin{align*}
&\int^m_0 \langle (L_0- \int^t_0 \alpha_0(s) \Phi'_0(s)ds)^{-1} , \Phi'(t) -\Phi'_0(t) \rangle dt \\
- &\int^m_0 \langle \xi'''(\Phi_0(t)) \odot (\Phi(t)-\Phi_0(t)) \odot \Phi'(t) + \xi''(\Phi_0(t)) \odot (\Phi'(t)-\Phi'_0(t)), \int^t_0 \alpha_0(s) \Phi_0'(s)ds \rangle dt \\
=& \int^m_0 \alpha_0(t) \langle \Phi(t) -\Phi_0(t), (L_0- \int^t_0 \alpha_0(s) \Phi'_0(s)ds)^{-1}  \Phi_0'(t) (L_0- \int^t_0 \alpha_0(s) \Phi'_0(s)ds)^{-1} \rangle\\
 -&\int^m_0 \alpha_0(t) \langle \Phi(t) -\Phi_0(t),  \xi''(\Phi_0(t)) \odot \Phi'(t) \rangle dt\geq 0
\end{align*}
Hence, writing
\[ \mathcal Z(s)=(L_0-\int^s_0 \alpha_0(q) \Phi'_0(q) dq )^{-1},
\]
\begin{align*}
\int^m_0 (\alpha(t)-\alpha_0(t))\langle \Phi'(t), \xi'(\Phi(t))-\int^t_0  \mathcal Z(s) \Phi_0'(s) \mathcal Z(s) ds \rangle dt \geq 0.
\end{align*}

If $t$ is an isolated point in supp $\gamma_0$, then $\Phi'(t)$ can be any symmetric matrix, which implies that $\bar{G}(t)=0$. If $t$ is not isolated, as $\Phi$ is continuously differentiable and Lipschitz, we can get the same conclusion based on approximation.  Last, define $g(t):=\int^m_t \bar{g}(s) ds$, then $g(t)$ satisfies $\min_{u \in [0,m]} g(u) \geq 0$ and $\gamma_0(S)=\gamma_0([0,m))$ where $S:= \{ u \in [0,m) | g(u) =0 \}$.
The converse direction can be proved by the uniqueness of the minimizer.
\end{proof}

\begin{proposition}
The model is replica symmetric at zero temperature if and only if 
\begin{equation*}
\xi'(Q)+\vec{h}\vec{h}^T \geq \xi''(Q) \odot Q
\end{equation*}
In this case, the minimizer $(L_0,\alpha_0,\Phi_0)$ is given by 
\begin{eqnarray*}
L_0=Q^{\frac{1}{2}}(Q^{\frac{1}{2}} (\xi'(Q)+\vec{h}\vec{h}^T) Q^{\frac{1}{2}})^{-\frac{1}{2}}Q^{\frac{1}{2}}, \alpha_0=0 \text{ and } \Phi_0=\frac{t}{m}Q.
\end{eqnarray*}
\end{proposition}

\begin{proof}
Firstly, if the model is replica symmetric at zero temperature, then
\begin{eqnarray*}
\xi'(Q)+\vec{h}\vec{h}^T- L_0^{-1} Q L_0^{-1}=0.
\end{eqnarray*}
Also as there is no point in supp $\gamma_0$, we can define $\Phi$ by $\Phi(t)=\frac{t}{m}Q$, and hence $\Phi'(t)=\frac{1}{m}Q.$

We prove this proposition by contradiction.  Assume $\xi'(Q)+\vec{h}\vec{h}^T < \xi''(Q) \odot Q$. Then,
\begin{eqnarray*}
\bar{g}'(m)&=&\langle \xi''(Q)\odot \Phi'(m)- L_0 ^{-1} \Phi'(m) L_0^{-1} ,\Phi'(m) \rangle\\
&=&\frac{1}{m^2} \langle \xi''(Q)\odot Q-\xi'(Q)-\vec{h}\vec{h}^T , Q \rangle >0,
\end{eqnarray*}
which implies that there exists $s_0 \in (0,m)$, such that $\bar{g}'(s) >0$ for all $s \in [s_0,m]$. Therefore for all $s \in [s_0,m)$,we obtain $\bar{g}(s) <0$ and $g(s)<0$, which leads to a contradiction.

Conversely, if $\xi'(Q)+\vec{h}\vec{h}^T \geq \xi''(Q) \odot Q$,
let \begin{eqnarray*}
L_0=Q^{\frac{1}{2}}(Q^{\frac{1}{2}} (\xi'(Q)+\vec{h}\vec{h}^T) Q^{\frac{1}{2}})^{-\frac{1}{2}}Q^{\frac{1}{2}}, \alpha_0=0 \text{ and } \Phi_0=\frac{t}{m}Q.
\end{eqnarray*}
Then $\int^m_0 (L_0-\int^s_0 \alpha_0(q) \Phi'(q) dq)^{-1} \Phi'(s)  (L_0-\int^s_0 \alpha_0(q) \Phi'(q) dq)^{-1}ds=L_0^{-1}QL_0^{-1}=\xi'(Q)+\vec{h}\vec{h}^T$.
Futhermore,
\begin{align*}
\bar{g}(t)&=\langle \xi'(\Phi(t))-\int^t_0 (L_0-\int^s_0 \alpha_0(q) \Phi'(q) dq)^{-1} \Phi'(s)  (L_0-\int^s_0 \alpha_0(q) \Phi'(q) dq)^{-1}ds,\Phi'(t) \rangle , \bar{g}(m)\\ 
&=0
\end{align*}
and
\begin{eqnarray*}
\bar{g}'(t)=  \frac{1}{m^2} \langle \xi''(\Phi(t))\odot Q- L_0^{-1} Q  L_0^{-1},Q \rangle &\leq& \frac{1}{m^2} \langle \xi''(Q)\odot Q- L_0^{-1} Q  L_0^{-1} ,Q\rangle \\
&=&\frac{1}{m^2 } \langle \xi''(Q)\odot Q-\xi'(Q)-\vec{h}\vec{h}^T,Q \rangle \leq 0,
\end{eqnarray*}
which implies that $g(u) >0$ for $u \in (0,m)$. Since $S=\emptyset$ and $\gamma_0(S)=0=\gamma([0,m))$, we conclude that $(L_0,\alpha_0,\Phi_0)$ is minimizer. which means that the model is replica symmetric at zero temperature.
\end{proof}

{\bf Acknowledgement:} Both authors would like to thank Wei-Kuo Chen and Justin Ko for useful insights on an early version of this manuscript. Y. Z. thanks Justin Ko for explaining the proof of the variational problem \eqref{eq:CSFormula} in detail.


\begin{thebibliography}{1}
\bibitem{AS2}
	M.~Aizenman, R.~Sims, and S.~L. Starr.
	\newblock Extended variational principle for the Sherrington-Kirkpatrick
	spin-glass model.
	\newblock {\em Phys. Rev. B}, 68:214403, Dec 2003.
	
	\bibitem{AB}
A.~Auffinger and G.~Ben~Arous.
\newblock Complexity of random smooth functions on the high-dimensional sphere.
\newblock {\em Ann. Probab.}, 41(6):4214--4247, 2013.

\bibitem{ABC}
A.~Auffinger, G.~Ben~Arous, and J.~{\v{C}}ern{\'y}.
\newblock Random matrices and complexity of spin glasses.
\newblock {\em Comm. Pure Appl. Math.}, 66(2):165--201, 2013.
	
	\bibitem{ParisiMeasure}
A. Auffinger, and W.-K. Chen.  The Parisi formula has a unique minimizer. {\em Comm. Math. Phys.},  335, no. 3, 1429--1444, (2015).
	
	\bibitem{parisigroundstate2}
A.~Auffinger, and W.-K. Chen.
	\newblock Parisi formula for the ground state energy in the mixed {$p$}-spin
	model.
	\newblock  {\em Ann. Probab.}, 45(6B):4617--4631, 2017.
	
\bibitem{1} A. Auffinger, and W.-K. Chen. On properties of Parisi measures, {\em Probab. Theory and Rel. Fields.}, Vol. 161, Issue 3, pp 817-850, 2015.

\bibitem{AJ} A. Auffinger, , and A. Jagannath. Thouless-Anderson-Palmer equations for generic $p$-spin glasses. {\em Ann. Probab.} 47 (2019), no. 4, 2230--2256. 


	\bibitem{CASS}
	W.-K. Chen.
	\newblock  The {A}izenman-{S}ims-{S}tarr scheme and {P}arisi formula for mixed
	{$p$}-spin spherical models.
	\newblock  {\em Electron. J. Probab.}, 18:no. 94, 14, 2013.
	
	\bibitem{chenchaos}
	W.-K. Chen, H.-W. Hsieh, C.-R. Hwang, and Y.-C. Sheu.
	\newblock  Disorder chaos in the spherical mean-field model.
	\newblock  {\em J. Stat. Phys.}, 160(2):417--429, 2015.
	
	\bibitem{chen2017temperature}
	W.-K. Chen, and D.~Panchenko.
	\newblock  Temperature chaos in some spherical mixed {$p$}-spin models.
	\newblock {\em J. Stat. Phys.}, 166(5):1151--1162, 2017.
	
	\bibitem{parisigroundstate1}
	W.-K. Chen, and A.~Sen.
	\newblock  Parisi formula, disorder chaos and fluctuation for the ground state
	energy in the spherical mixed {$p$}-spin models.
	\newblock  {\em Comm. Math. Phys.}, 350(1):129--173, 2017.

\bibitem{2} A. Crisanti and H.-J. Sommers. The spherical p-spin interaction spin glass model: the statics, {\em Zeitschrift fur Physik B Condensed Matter}, 87(3):341-354, 1992.

	\bibitem{replicaonandoff}
	S.~Franz, G.~Parisi, and M.~Virasoro.
	\newblock The replica method on and off equilibrium.
	\newblock  {\em  http://dx.doi.org/10.1051/jp1:1992115}, 2, 10 1992.
	
	\bibitem{freeenergycostforultrametricity}
	S.~Franz, G.~Parisi, and M.~A. Virasoro.
	\newblock Free-energy cost for ultrametricity violations in spin glasses.
	\newblock{\em  Europhysics Letters ({EPL})}, 22(6):405--411, May 1993.


		\bibitem{Guerra}
  		 Guerra, F.: Broken replica symmetry bounds in the mean field spin glass model. {\it Comm. Math. Phys.}, {\bf 233}, no. 1, (2003).
%  		

\bibitem{dualitygroundstate}
	A.~Jagannath, and I.~Tobasco.
	\newblock Low temperature asymptotics of spherical mean field spin glasses.
	\newblock {\em  Comm. Math. Phys.,} 352(3):979--1017, 2017.
	
	\bibitem{phasediagram}
	A.~Jagannath, and I.~Tobasco.
	\newblock Some properties of the phase diagram for mixed {$p$}-spin glasses.
	\newblock  {\em Probab. Theory Related Fields}, 167(3-4):615--672, 2017.



\bibitem{JK1} J. Ko. \textit{The Crisanti-Sommers Formula for Spherical Spin Glasses with Vector Spins}, arXiv:1911.04355, 2019.
\bibitem{Ko} J. Ko. \textit{Free energy of multiple systems of spherical spin glasses with constrained overlaps}, arXiv:1806.09772, 2018.


\bibitem{panchenko2015free}
	D.~Panchenko.
	\newblock The free energy in a multi-species {S}herrington-{K}irkpatrick model.
	\newblock {\em  Ann. Probab.,} 43(6):3494--3513, 2015.
	
\bibitem{MPV}
 		M. M\'{e}zard, G. Parisi, and M. Virasoro.: Spin glass theory and beyond. World Scientific, {\bf 9}, Singapore, (2004).

\bibitem{4}
	D.~Panchenko.
	\newblock Free energy in the mixed {$p$}-spin models with vector spins.
	\newblock  {\em Ann. Probab.,} 46(2):865--896, 2018.
	
\bibitem{5}
	D.~Panchenko.
	\newblock Free energy in the {P}otts spin glass.
	\newblock  {\em Ann. Probab., }46(2):829--864, 2018.

\bibitem{PTSPHERE}
	D.~Panchenko, and M.~Talagrand.
	\newblock  On the overlap in the multiple spherical {SK} models.
	\newblock  {\em Ann. Probab.}, 35(6):2321--2355, 2007.

\bibitem{SK}
Sherrington, D., Kirkpatrick, S.:
\newblock Solvable model of a spin-glass.
\newblock {\em Phys. Rev. Lett.}, {\bf 35}(26):1792--1796, (1975).

\bibitem{6}M. Talagrand. Free energy of the spherical mean field model, {\em Probab. Theory Related Fields}, 134(3):339-382, 2006.
	\bibitem{Talagrand}
  		M. Talagrand. The Parisi formula. {\it Ann. of Math. (2)}, {\bf 163}, no. 1, 221--263, (2006).

	
	

\end{thebibliography}
\end{document}